\documentclass[11pt,reqno]{amsart}

\usepackage[dvipsnames]{xcolor}
\usepackage{amsmath,amssymb,amsthm,amsfonts,enumerate,tikz,bm}
\usepackage{mathrsfs}
\usetikzlibrary{matrix,arrows,positioning,calc}
\usepackage[all]{xy}
\usepackage[bookmarksnumbered,colorlinks]{hyperref}
\usepackage{url}
\numberwithin{equation}{section}
\usepackage{graphicx}
\usepackage[T1]{fontenc}
\usepackage[hang,flushmargin]{footmisc}

\theoremstyle{definition}
\newtheorem{definition}{Definition}[section]
\newtheorem{remark}[definition]{Remark}
\newtheorem{example}[definition]{Example}

\theoremstyle{plain}
\newtheorem{theorem}[definition]{Theorem}
\newtheorem{proposition}[definition]{Proposition}
\newtheorem{lemma}[definition]{Lemma}
\newtheorem{corollary}[definition]{Corollary}

\theoremstyle{remark}

\usepackage{paralist}

\newcommand{\Hom}{\mathsf{Hom}}
\newcommand{\Ext}{\mathsf{Ext}}
\newcommand{\Tor}{\mathsf{Tor}}
\newcommand{\Ch}{\mathsf{Ch}}

\newcommand{\too}{\longrightarrow}

\newcommand{\Mod}{\mathsf{Mod}}

\newcommand{\Ker}{{\rm Ker}}

\newcommand{\op}{{}^{\rm op}}

\newcommand{\dson}{\displaystyle\operatorname*}

\makeatletter
\def\@seccntformat#1{%
  \protect\textup{\protect\@secnumfont
    \ifnum\pdfstrcmp{section}{#1}=0 \scshape\bfseries\fi
    \ifnum\pdfstrcmp{subsection}{#1}=0 \bfseries\fi
    \csname the#1\endcsname
    \protect\@secnumpunct
  }%
}
\makeatother

\begin{document}

\title{Model structures and relative Gorenstein flat modules \\ and chain complexes}
\thanks{The first author is supported by research projects from the Ministerio de Econom\'ia y Competitividad of Spain
(MTM2013-46837-P) and by the Fundaci\'on S\'eneca of Murcia (19880/GERM/15)}
\author{Sergio Estrada}
\address[S. Estrada]{Departamento de Matem\'aticas. Universidad de Murcia. Murcia 30100. ESPA\~{N}A}
\email{sestrada@um.es}

\author{Alina Iacob}
\address[A. Iacob]{Department of Mathematical Sciences. Georgia Southern University. Statesboro (GA) 30460-8093. USA}
\email{aiacob@GeorgiaSouthern.edu}

\author{Marco A. P\'erez}
\address[M. A. P\'erez]{Instituto de Matem\'atica y Estad\'istica ``Prof. Ing. Rafael Laguardia''. Universidad de la Rep\'ublica. Montevideo 11300. URUGUAY}
\email{mperez@fing.edu.uy}

\baselineskip=14pt
\maketitle

\begin{abstract}
A recent result by J. {\v{S}}aroch and J. {\v{S}}\v{t}ov\'{\i}\v{c}ek asserts that there is a unique abelian model structure on the category of left $R$-modules, for any associative ring $R$ with identity, whose (trivially) cofibrant and (trivially) fibrant objects are given by the classes of Gorenstein flat (resp., flat) and cotorsion (resp., Gorenstein cotorsion) modules. In this paper, we generalise this result to a certain relativisation of Gorenstein flat modules, which we call Gorenstein $\mathcal{B}$-flat modules, where $\mathcal{B}$ is a class of right $R$-modules. Using some of the techniques considered by {\v{S}}aroch and {\v{S}}\v{t}ov\'{\i}\v{c}ek, plus some other arguments coming from model theory, we determine some conditions for $\mathcal{B}$ so that the class of Gorenstein $\mathcal{B}$-flat modules is closed under extensions. This will allow us to show approximation properties concerning these modules, and also to obtain a relative version of the model structure described before. Moreover, we also present and prove our results in the category of complexes of left $R$-modules, study other model structures on complexes constructed from relative Gorenstein flat modules, and compare these models via computing their homotopy categories. 
\end{abstract}


\pagestyle{myheadings}
\markboth{\rightline {\scriptsize S. Estrada, A. Iacob and M. A. P\'{e}rez}}
         {\leftline{\scriptsize Model structures and relative Gorenstein flat modules and chain complexes}}


\section*{\textbf{Introduction}}

In a recent paper \cite[Theorem 3.3]{GillespieGF}, James Gillespie constructs a hereditary abelian model structure on the category $\Mod(R)$ of left $R$-modules over a coherent ring, where the cofibrant objects are given by the class of Gorenstein flat modules, and the fibrant objects by the class of cotorsion modules. The method used by the author to obtain this new model structure is described in the general setting of abelian categories, in another paper of his (see \cite[Main Theorem 1.2]{GillespieTriple}). 

One of our goals is to extend Gillespie's \cite[Theorem 3.3]{GillespieGF} to the categories of modules and chain complexes over an arbitrary ring. This result is a consequence of a more general approach dealing with Gorenstein flat modules relative to a class $\mathcal{B}$ of right $R$-modules. This notion carries over to the category of chain complexes, and as a consequence, we shall also obtain the Gorenstein flat model structure on $\Ch(R)$.

We introduce the notion of Gorenstein flat left $R$-modules relative to a class $\mathcal{B} \subseteq \Mod(R\op)$ in Section \ref{sec:modules}. When the class $\mathcal{B}$ contains the injective modules in $\Mod(R\op)$, every Gorenstein $\mathcal{B}$-flat module is, in particular, Gorenstein flat. When $\mathcal{B}$ is the class of injectives, the two classes coincide. We prove some general properties and equivalent characterisations of the Gorenstein $\mathcal{B}$-flat modules. Among other things, we show that the class $\mathcal{GF}_{\mathcal B}(R)$ of Gorenstein $\mathcal{B}$-flat modules is a precovering Kaplansky class. Then we consider the question of closure under extensions for $\mathcal{GF}_{\mathcal B}(R)$. For the absolute Gorenstein flat modules, this question was recently settled in the affirmative over arbitrary rings, by J. {\v{S}}aroch and J. {\v{S}}\v{t}ov\'{\i}\v{c}ek \cite{SarochStovicek}. We give sufficient conditions on the class $\mathcal{B}$ for $\mathcal{GF}_{\mathcal B}(R)$ being closed under extensions. We prove that if $\mathcal{B}$ is a semi-definable class (that is, closed under products and containing an elementary cogenerator of its definable closure - see Section~\ref{sec:modules} for details) then the class of relative Gorenstein flat modules is a left orthogonal class, and therefore it is closed under extensions. Moreover, for any semi-definable class $\mathcal{B}$, the class $\mathcal{GF}_{\mathcal B}(R)$ is closed under direct limits, so it is a covering class. We prove that in this case ($\mathcal{B}$ is semi-definable), $\mathcal{GF}_{\mathcal B}(R)$ is the left half of a hereditary complete cotorsion pair. 

The results mentioned above are also presented in the setting provided by the category $\Ch(R)$ of chain complexes. The definition of Gorenstein flat complexes relative to a class of complexes in $\Ch(R\op)$ will be presented considering a particular tensor product of complexes introduced by J. R. Garc\'ia Rozas in \cite{jrgr}. We shall also study complexes of relative Gorenstein flat modules, and provide a characterisation for them as Gorenstein flat complexes relative to a certain class of complexes in $\Ch(R\op)$. We prove that, for any class of right $R$-modules $\mathcal{B}$, the class of Gorenstein $\Ch(\mathcal{B})$-flat complexes is contained in that of complexes of Gorenstein $\mathcal{B}$-flat modules, where $\Ch(\mathcal{B})$ denotes the class of complexes of modules in $\mathcal{B}$. But, as Remark \ref{rem:strict_inclusion} shows, this inclusion is, in general, a strict one.

Model structures are the main topic of Section \ref{sec:stable}. Assuming closure under extensions for relative Gorenstein flat modules, we shall be able to construct a new model structure on $\Mod(R)$ in which these modules are the cofibrant objects. As a consequence, we obtain Gillespie's Gorenstein flat model structure on $\Mod(R)$ when $R$ is a coherent ring (see \cite{GillespieGF}). Part of our generalisation has to do with one of our corollaries showing that Gillespie's model exists over arbitrary rings. We also prove the existence of the relative Gorenstein flat model structure on the category of complexes. We point out in Remark \ref{rem:scheme} that our method to get the model differs from the one showed in \cite{SarochStovicek} in the absolute case, as ours only relies on the closure under extensions of the class $\mathcal{GF}_{\mathcal B}(R)$.

Section \ref{sec:comparison} is devoted to comparing several model structures on chain complexes associated to relative Gorenstein flat modules. Quillen adjunctions between these models are obtained, and for some of them, their homotopy categories will form an arrangement of triangulated categories and functors known as recollement.


\subsection*{Notation}

Most of the categories considered in this paper are Grothendieck categories. We shall use the symbol $\simeq$ to represent that two objects in a category are isomorphic, while natural isomorphisms between functors will be denoted using $\cong$. Monomorphisms and epimorphisms in a category may sometimes be denoted by $\rightarrowtail$ and $\twoheadrightarrow$, respectively. Subcategories of a given category are always assumed to be full. 

In what follows, $R$ will be an associative ring with unit $1 \in R$. We denote the categories of left and right $R$-modules by $\Mod(R)$ and $\Mod(R\op)$, respectively. In the case where $R = \mathbb{Z}$ is the ring of integers, we shall denote the category $\Mod(\mathbb{Z}) = \Mod(\mathbb{Z}^{\rm op})$ of abelian groups by $\mathsf{Ab}$. The categories of chain complexes of left and right $R$-modules will be denoted by $\Ch(R)$ and $\Ch(R\op)$, respectively. By ``module'' we usually mean a left $R$-module, and complexes of left $R$-modules will be simply referred to as ``complexes''. 

Recall that by a complex $X \in \Ch(R)$ one means a collection $X = (X_m, \partial^X_m \colon X_m \to X_{m-1})_{m \in \mathbb{Z}}$ of modules $X_m \in \Mod(R)$ and $R$-homomorphisms $\partial^X_m \colon X_m \to X_{m-1}$, called \emph{differentials}, such that $\partial^X_m \circ \partial^X_{m+1} = 0$ for every $m \in \mathbb{Z}$. The cycles and boundaries of $X$ are defined as the modules $Z_m(X) := {\rm Ker}(\partial^X_m)$ and $B_m(X) := {\rm Im}(\partial^X_{m+1})$, respectively. In some situations, we shall need to consider the category $\Ch(\Ch(R))$ of complexes of complexes, whose objects will usually be denoted as $X_\bullet$, that is, $X_m \in \Ch(R)$ for every $m \in \mathbb{Z}$.

Disks and spheres will be recurrent examples of chain complexes in this paper. Recall that given a module $M \in \Mod(R)$ and an integer $m \in \mathbb{Z}$, the \emph{$m$-th disk centred at $M$} is the complex $D^m(M)$ with $M$ at the $m$-th and $(m-1)$-th positions, and $0$ elsewhere, with differential $M \to M$ given by the identity. The \emph{$m$-th sphere centred at $M$} is the complex $S^m(M)$ with $M$ at the $m$-th position and $0$ elsewhere.


\section{\textbf{Preliminaries}}

In this section we recall the background material that will be used in the sequel. We also prove some results concerning duality pairs and definable classes of modules and chain complexes.


\subsection*{Tensor product of complexes}

At some point we shall define relative Gorenstein flat complexes as cycles of complexes of flat complexes satisfying certain acyclicity condition. The latter is specified in terms of tensor products of complexes. There are two well known tensors in this setting, specified below.

Consider the usual tensor product of modules over a ring:
\[
- \displaystyle\operatorname*{\otimes}_R - \colon \Mod(R\op) \times \Mod(R) \too \mathsf{Ab}.
\] 
Now let $X \in \Ch(R\op)$ and $Y \in \Ch(R)$. Recall from \cite[Section 2.1]{jrgr} that the (usual) \emph{tensor product} of $X$ and $Y$ is the chain complex $X \otimes^. Y \in \Ch(\mathbb{Z})$ of abelian groups defined by
\[
(X \otimes^. Y)_m := \bigoplus_{k \in \mathbb{Z}} X_k \displaystyle\operatorname*{\otimes}_R Y_{m-k}
\]
whose differential maps are given by 
\[
x \otimes y \mapsto \partial^X_k(x) \otimes y + (-1)^k x \otimes \partial^Y_{m-k}(y)
\] 
if $x \in X_k$ and $y \in Y_{m-k}$. There is an associated hom functor 
\[
\mathcal{H}{\rm om}(-,-) \colon \Ch(R) \times \Ch(R) \longrightarrow \Ch(\mathbb{Z})
\] 
which maps every pair of complexes $X, Y \in \Ch(R)$ to the complex $\mathcal{H}{\rm om}(X,Y)$ given by 
\[
\mathcal{H}{\rm om}(X,Y)_m := \prod_{k \in \mathbb{Z}} \Hom_R(X_k, Y_{m+k})
\] 
and with differentials 
\[
f \mapsto \partial^Y_{m+k} \circ f_k - (-1)^m f_{k-1} \circ \partial^X_k.
\] 
The \emph{modified tensor product} of $X$ and $Y$ is the complex $X \otimes Y \in \Ch(\mathbb{Z})$ of abelian groups defined by
\[
(X \otimes Y)_m := \frac{(X \otimes^. Y)_m}{B_m(X \otimes^. Y)},
\] 
with differentials given by 
\[
x \otimes y \mapsto \partial^X(x) \otimes y
\] 
(see \cite[Section 4.2]{jrgr} for details). This tensor product has also an associated hom defined from $\mathcal{H}{\rm om}(-,-)$. Namely, there is a functor 
\[
\underline{\Hom}(-,-) \colon \Ch(R) \times \Ch(R) \longrightarrow \Ch(\mathbb{Z})
\] 
mapping every pair of complexes $X, Y \in \Ch(R)$ to a complex $\underline{\Hom}(X,Y)$ given by 
\[
\underline{\Hom}(X,Y)_m := Z_m(\mathcal{H}{\rm om}(X,Y))
\] 
with differentials 
\[
f \mapsto (-1)^m \partial^Y_{m+k} \circ f_k.
\]


\subsection*{Approximations}

One of the goals of the present paper is to construct approximations by relative Gorenstein flat modules.

Given a Grothendieck category $\mathcal{G}$, a subcategory $\mathcal{A}$ of objects of $\mathcal{G}$, and an object $M \in \mathcal{G}$, recall that a morphism $\varphi \colon A \to M$ is called an \emph{$\mathcal{A}$-precover} (or a \emph{right $\mathcal{A}$-approximation}) of $M$ if $A \in \mathcal{A}$ and if for every morphism $\varphi \colon A' \to M$ there exists $h \colon A' \to A$ such that $\varphi' = \varphi \circ h$. If in addition $h$ is an automorphism of $A$ in the case where $A' = A$ and $\varphi' = \varphi$, then $\varphi$ is called \emph{$\mathcal{A}$-cover} of $M$. The notions of \emph{$\mathcal{A}$-preenvelopes} (\emph{left $\mathcal{A}$-approximations}) and \emph{$\mathcal{A}$-envelopes} are dual. 

Recall also that given a subcategory $\mathcal{A} \subseteq \mathcal{G}$, the \emph{right orthogonal complement} of $\mathcal{A}$ is defined as the class
\[
\mathcal{A}^\perp := \{ N \in \mathcal{G} \mbox{ : } \Ext^1_{\mathcal{G}}(A,N) = 0 \mbox{ for every } A \in \mathcal{A} \}.
\] 
The \emph{left orthogonal complement} ${}^\perp\mathcal{A}$ is defined similarly. An $\mathcal{A}$-precover $\varphi \colon A \to M$ of $M$ is \emph{special} if $\varphi$ is epic and ${\rm Ker}(\varphi) \in \mathcal{A}^\perp$. \emph{Special $\mathcal{A}$-preenvelopes} are defined similarly. 

A  subcategory $\mathcal{A} \subseteq \mathcal{G}$ is \emph{precovering}, \emph{covering} or \emph{special precovering} if every module has an $\mathcal{A}$-precover, an $\mathcal{A}$-cover or a special $\mathcal{A}$-precover, respectively. \emph{Preenveloping}, \emph{enveloping} and \emph{special preenveloping} classes are dual. 

A natural source to obtain left and right approximations is by means of cotorsion pairs. Two  subcategories $\mathcal{A}, \mathcal{B} \subseteq \mathcal{G}$ form a \emph{cotorsion pair} $(\mathcal{A,B})$ in $\mathcal{G}$ if $\mathcal{A} = {}^\perp\mathcal{B}$ and $\mathcal{B} = \mathcal{A}^\perp$. A cotorsion pair $(\mathcal{A,B})$ in $\mathcal{G}$ is \emph{complete} if every object of $\mathcal{G}$ has a special $\mathcal{A}$-precover and a special $\mathcal{B}$-preenvelope. Similarly, $(\mathcal{A,B})$ is said to be \emph{perfect} if every object of $\mathcal{G}$ has an $\mathcal{A}$-cover and a $\mathcal{B}$-envelope. As examples, if $\mathcal{P}(R)$ and $\mathcal{I}(R)$ denote the classes of projective and injective modules, respectively, then $(\mathcal{P}(R),\Mod(R))$ and $(\Mod(R),\mathcal{I}(R))$ are clearly complete cotorsion pairs in $\Mod(R)$. Moreover, if $\mathcal{F}(R)$ denotes the class of flat modules and $\mathcal{C}(R) := (\mathcal{F}(R))^\perp$ is the class of cotorsion modules, then $(\mathcal{F}(R),\mathcal{C}(R))$ is a perfect cotorsion pair in $\Mod(R)$ (see \cite{EJ}, for instance). 

A subcategory $\mathcal{A} \subseteq \mathcal{G}$ is \emph{resolving} if it contains the projective objects of $\mathcal{G}$ and it is closed under extensions and under taking kernels of epimorphisms with domain and codomain in $\mathcal{A}$. \emph{Coresolving} subcategories are defined dually. A cotorsion pair $(\mathcal{A,B})$ in $\mathcal{G}$ is \emph{hereditary} if $\mathcal{A}$ is resolving and $\mathcal{B}$ is coresolving. 

A method to construct left and right approximations from a cotorsion pair $(\mathcal{A,B})$ consists in providing a cogenerating set for $(\mathcal{A,B})$. A result known as the Eklof and Trlifaj Theorem asserts that every cotorsion pair (in any Grothendieck category with enough projectives) \emph{cogenerated} by a set is complete (see Hovey's \cite[Corollary 6.8]{Hovey}, for instance). This means the existence of a set $\mathcal{S} \subseteq \mathcal{A}$ such that $\mathcal{B} = \mathcal{S}^\perp$. 
 
Another useful concept for finding approximations is that of a Kaplansky class. For instance, any Kaplansky class in a Grothendieck category closed under direct limits and direct products is preenveloping. Recall from Gillespie's \cite[Definition 5.1]{GillespieDegreewise} that a subcategory $\mathcal{A}$ of objects in $\mathcal{G}$ is a \emph{Kaplansky class} if there is a regular cardinal $\kappa$ for which the following condition is satisfied: For any subobject $M \subseteq A$, with $A \in \mathcal{A}$ nonzero and $M$ $\kappa$-generated, there exists a $\kappa$-presentable object $S \neq 0$ such that $M \subseteq S \subseteq A$ and $S, A / S \in \mathcal{A}$.


\subsection*{Duality pairs}

The notion of duality pairs will be useful to provide some characterisations for relative Gorenstein flat modules and complexes. We shall see that, given an exact complex $P$ of projective modules and a duality pair $(\mathcal{A,B})$, checking the acyclicity of $P$ with respect to $\mathcal{B} \otimes_R -$ is equivalent to checking the acyclicity of $P$ but with respect to $\Hom_R(-,\mathcal{A})$ instead.

Recall that the \emph{character module} of a left $R$-module $M \in \Mod(R)$ is defined as the right $R$-module $M^+ := \Hom_{\mathbb{Z}}(M,\mathbb{Q / Z})$. Character modules of right $R$-modules are defined in the same way. 

Two classes $\mathcal{A} \subseteq \mathsf{Mod}(R)$ and $\mathcal{B} \subseteq \mathsf{Mod}(R^{\rm op})$ form a \emph{duality pair} $(\mathcal{A,B})$ (in the sense of \cite{BGH}) provided that: 
\begin{enumerate}
\item $A \in \mathcal{A}$ if, and only if, $A^+ \in \mathcal{B}$, and 

\item $B \in \mathcal{B}$ if, and only if, $B^+ \in \mathcal{A}$. 
\end{enumerate}
For example, $(\mathcal{F}(R),\mathcal{I}(R\op))$ is a duality pair over $R$ provided that $R$ is a left Noetherian ring. Furthermore, if $\mathcal{I}_1(R\op)$ denotes the class of FP-injective right $R$-modules, then $(\mathcal{F}(R),\mathcal{I}_1(R\op))$ is also a duality pair for the case where $R$ is a left coherent ring.

Duality pairs of classes of complexes can be defined in a similar fashion. For, if $X \in \Ch(R)$ is a complex of left $R$-modules, its \emph{character complex} is defined, in the sense of \cite{jrgr}, as the complex of right $R$-modules $X^+ := \underline{\Hom}(X,D^1(\mathbb{Q / Z}))$.


\subsection*{Definable classes}

Some duality pairs can be constructed from definable classes, as we shall see in the last part of this section. Let us recall from \cite{PrestDefinable} the concept of definable classes in finitely accessible additive categories $\mathcal{G}$ with products, although we shall be interested in the cases where $\mathcal{G}$ is the category $\Mod(R)$ of modules or $\Ch(R)$ of chain complexes. 

Recall that a preadditive category $\mathcal{G}$ is \emph{finitely accessible} if it has direct limits and if the subcategory $\mathcal{G}^{\rm fp}$ of finitely presented objects of $\mathcal{G}$ is skeletally small. Recall also that $X \in \mathcal{G}$ is \emph{finitely presented} if the representable functor $\Hom_{\mathcal{G}}(X,-) \colon \mathcal{G} \longrightarrow \mathsf{Ab}$ preserves direct limits.  

Given a finitely accessible additive category $\mathcal{G}$ with products, a subcategory $\mathcal{D}$ is \emph{definable} if it is closed under products, direct limits and pure subobjects. A sequence 
\[
\rho \colon 0 \to X \to Y \to Z \to 0
\] 
in $\mathcal{G}$ is \emph{pure} if the induced sequence 
\[
\Hom_{\mathcal{G}}(L,\rho) \colon 0 \to \Hom_{\mathcal{G}}(L,X) \to \Hom_{\mathcal{G}}(L,Y) \to \Hom_{\mathcal{G}}(L,Z) \to 0
\] 
in $\mathsf{Ab}$ is exact for every finitely presented object $L \in \mathcal{G}^{\rm fp}$. In this situation, one says that $X$ is a \emph{pure subobject} of $Y$ and that $Z$ is a \emph{pure quotient} of $Y$. In the case $\mathcal{G} = \Ch(R)$, being pure is equivalent to saying that the induced sequence
\[
W \otimes \rho \colon 0 \to W \otimes X \to W \otimes Y \to W \otimes Z \to 0
\]
is exact for every complex of right $R$-modules $W$. The same applies to the case $\mathcal{G} = \Mod(R)$ in terms of $- \otimes_R -$. 

An \emph{elementary cogenerator} of a definable subcategory $\mathcal{D}$ is a pure-injective object $D_0 \in \mathcal{D}$ such that every $D \in \mathcal{D}$ is a pure subobject of some product of copies of $D_0$. In this case, we shall use the notation $\mathcal{D} = {\rm CoGen}_\ast(D_0)$. Here, \emph{pure-injective} means injective with respect to pure exact sequences. In \cite[Theorem 21.7]{PrestDefinable}, M. Prest proves that every definable subcategory of a finitely accessible additive category with products has an elementary cogenerator. The reader can also see \v{S}aroch's \cite[Lemma 5.3]{SarochTools} or Prest's \cite[Corollary 5.3.52]{PrestPurity} for the case of modules over a ring. 

Let us set some special classes of objects constructed from a class $\mathcal{B}$ of objects of a finitely accessible additive category $\mathcal{G}$ with products:
\begin{itemize}
\item We shall denote by $\langle \mathcal{B} \rangle$ the \emph{definable closure} of $\mathcal{B}$, that is, the smallest definable subcategory containing $\mathcal{B}$. For the purposes of the present paper, we shall be interested in subcategories $\mathcal{B}$ containing an elementary cogenerator $B_0$ of $\langle \mathcal{B} \rangle$.

\item ${\rm Prod}(\mathcal{B})$ will denote the subcategory of direct summands of direct products of objects in $\mathcal{B}$.

\item $\mathcal{B}^{\rm p}$ will stand for the subcategory of pure subobjects of objects in $\mathcal{B}$. 

\item By ${\rm PInj}(\mathcal{B})$ we shall mean the subcategory of pure-injective objects in $\mathcal{B}$.

\item In the cases where $\mathcal{G} = \Mod(R)$ or $\mathcal{G} = \Ch(R)$, $\mathcal{B}^+$ will be the class of objects isomorphic to objets of the form $B^+$ with $B \in \mathcal{B}$. 
\end{itemize}

As we mentioned earlier, definable classes are useful to construct duality pairs. Let us specify this in detail below. We shall need the following lemma. Its proof will be commented later in a separate appendix at the end of the paper, since it involves some notions from model theory.

\begin{lemma}\label{lem:Prest}
Let $\mathcal{D}$ be a definable class of modules or chain complexes. Then, $D \in \mathcal{D}$ if, and only if, $D^{++} \in \mathcal{D}$.
\end{lemma}

\begin{theorem}\label{theo:definable_duality_pair}
Let $\mathcal{B} \subseteq \Mod(R\op)$ be a class of right $R$-modules (or complexes of right $R$-modules). Then, the following conditions hold true.
\begin{enumerate}
\item There exists a duality pair $(\mathcal{A},\langle \mathcal{B} \rangle)$, where $\mathcal{A}$ is also a definable class given by ${\rm Prod}(\langle \mathcal{B} \rangle^+)^{\rm p}$. 

\item Moreover, if $B_0$ is an elementary cogenerator of $\langle \mathcal{B} \rangle$, then the duality pair $(\mathcal{A},\langle\mathcal{B}\rangle)$ can be written as $(\langle B^+_0 \rangle,\langle B_0 \rangle)$. 
\end{enumerate}
\end{theorem}

\begin{proof}
Let us focus only on the case of modules. The proof for complexes follows in the same way. We split the proof into three parts:
\begin{enumerate}
\item[(i)] $\mathcal{A}$ is definable: First note that it is clear that $\mathcal{A}$ is closed under pure submodules. Now consider a family of modules $\{ A_i \in \mathcal{A} \mbox{ : } i \in I \}$.

Let us show that $\prod_{i \in I} A_i \in \mathcal{A}$. For each $i \in I$, we have a pure embedding $\rho_i \colon A_i \hookrightarrow X_i$ where $X_i \in {\rm Prod}(\langle \mathcal{B} \rangle^+)$. We may assume that $X_i \simeq \prod_{j \in J_i} B^+_j$ with each $B_j \in \langle \mathcal{B} \rangle$. Since the product of pure embeddings is again pure, we have that $\prod_{i \in I} A_i$ is a pure submodule of $\prod_{i \in I} X_i$, which in turn is a direct product of modules in $\langle \mathcal{B} \rangle^+$. It follows that $\prod_{i \in I} A_i \in \mathcal{A}$. 

Finally, suppose that $I$ is a directed set. We check $\varinjlim_{i \in I} A_i \in \mathcal{A}$. Consider the canonical pure epimorphism 
\[
\bigoplus_{i \in I} A_i \twoheadrightarrow \varinjlim_{i \in I} A_i.
\] 
Then, we have that $(\varinjlim_{i \in I} A_i)^+$ is a direct summand of $(\bigoplus_{i \in I} A_i)^+ \simeq \prod_{i \in I} A^+_i$. For each $i \in I$, we have isomorphisms 
\[
X_i \simeq \prod_{j \in J_i} B^+_j \simeq (\bigoplus_{i \in J_i} B_j)^+.
\] 
It follows that $\bigoplus_{i \in J_i} B_j \in \langle \mathcal{B} \rangle$, and so $X_i \in \langle \mathcal{B} \rangle^+$. Then, $X^+_i \in \langle \mathcal{B} \rangle^{++}$, and since $\langle \mathcal{B} \rangle$ is definable, we have by Lemma~\ref{lem:Prest} that $X^+_i \in \langle \mathcal{B} \rangle$. This in turn implies that $A^+_i \in \langle \mathcal{B} \rangle$ since $A^+_i$ is a direct summand of $X^+_i$, and then $(\bigoplus_{i \in I} A_i)^+ \simeq \prod_{i \in I} A^+_i \in \langle \mathcal{B} \rangle$. Hence, $(\varinjlim_{i \in I} A_i)^+ \in \langle \mathcal{B} \rangle$, and so we can deduce that $\varinjlim_{i \in I} A_i \in \mathcal{A}$ since $\varinjlim_{i \in I} A_i$ is a pure submodule of $(\varinjlim_{i \in I} A_i)^{++} \in \langle \mathcal{B} \rangle^+$. 

\item[(ii)] $A \in \mathcal{A}$ if, and only if, $A^+ \in \langle \mathcal{B} \rangle$: The ``if'' part is clear. Now let us suppose that $A \in \mathcal{A}$. Then, $A$ is a pure submodule of a module $X \in {\rm Prod}(\langle \mathcal{B} \rangle^+)$. We may assume that $X \simeq \prod_{i \in I} B^+_i$ with $B_i \in \mathcal{B}$ for every $i \in I$. Thus, we have a pure embedding $A \hookrightarrow \prod_{i \in I} B^+_i$, and so $A^+$ is a direct summand of 
\[
X^+ \simeq (\prod_{i \in I} B^+_i)^+ \simeq (\bigoplus_{i \in I} B_i)^{++},
\] 
where $\bigoplus_{i \in I} B_i \in \langle \mathcal{B} \rangle$. Using Lemma~\ref{lem:Prest} again, we have that $X^+ \in \langle \mathcal{B} \rangle$, and hence $A^+ \in \langle \mathcal{B} \rangle$.  

\item[(iii)] The equivalence $Y \in \langle \mathcal{B} \rangle$ if, and only if, $Y^+ \in \mathcal{A}$ is similar to (2). 
\end{enumerate} 

So far we have proved that there is a duality pair $(\mathcal{A},\langle \mathcal{B} \rangle)$ with $\mathcal{A} = {\rm Prod}(\langle \mathcal{B} \rangle^+)^{\rm p}$. 

Now for the second part, note that the equality $\langle \mathcal{B} \rangle = \langle B_0 \rangle$ is straightforward. So let us first show the inclusion $\mathcal{A} \subseteq \langle B^+_0 \rangle$. It suffices to show that $\langle \mathcal{B} \rangle^+ \subseteq \langle B^+_0 \rangle$. Suppose we are given $N \in \langle \mathcal{B} \rangle$. Since $\langle \mathcal{B} \rangle = {\rm CoGen}_\ast(B_0)$, there exists a set $I$ such that $N$ is a pure submodule of the product $B^I_0$ of copies of $B_0$. On the other hand, there is a pure monomorphism $B_0 \hookrightarrow B^{++}_0$. Since any direct product of pure monomorphisms is pure, we then have a pure monomorphism $B^I_0 \hookrightarrow (B^{++}_0)^I$. It follows that $N$ is a pure submodule of $(B^{++}_0)^I$. Note also that $(B^{++}_0)^I \simeq ((B^+_0)^{(I)})^+$. Thus, we have a split (and so pure) epimorphism $((B^+_0)^{(I)})^{++} \twoheadrightarrow N^+$ where $(B^+_0)^{(I)} \in \langle B^+_0 \rangle$. Finally, by Lemma~\ref{lem:Prest} we have that $((B^+_0)^{(I)})^{++} \in \langle B^+_0 \rangle$, which in turn implies that $N^+ \in \langle B^+_0 \rangle$. Hence, the containment $\mathcal{A} \subseteq \langle B^+_0 \rangle$ follows. 

For the converse containment $\langle B^+_0 \rangle \subseteq \mathcal{A}$, it suffices to show that $B^+_0 \in \mathcal{A}$, but this follows by the definition of $\mathcal{A}$ as the class ${\rm Prod}(\langle \mathcal{B} \rangle^+)$. 
\end{proof}

\begin{remark} \
\begin{enumerate}
\item The module case for the first part of Theorem~\ref{theo:definable_duality_pair} can also be deduced from Mehdi and Prest's \cite[Corollary 4.6]{MehdiPrest}, although for pedagogical reasons we preferred to provide a more algebraic proof. Indeed, it suffices to note that every duality pair of definable classes of modules is a \emph{dual pair} in the sense of \cite[Section 1]{MehdiPrest}. For, one just needs to check the following conditions for any duality pair $(\mathcal{A,B})$ where $\mathcal{A} \subseteq \Mod(R)$ and $\mathcal{B} \subseteq \Mod(R\op)$ are definable:
\begin{itemize}
\item[(i)] $\mathcal{B}$ is closed under direct summands: This follows since direct summands are particular instances of pure submodules. 

\item[(ii)] $\mathcal{B}$ is closed under pure-injective envelopes: This means that the pure-injective envelope ${\rm PE}(B)$ of any module $B \in \mathcal{B}$ is also in $\mathcal{B}$.

Let $B \in \mathcal{B}$ and $B \hookrightarrow {\rm PE}(B)$ be its pure-injective envelope. Consider also the pure embedding $B \hookrightarrow B^{++}$. Since $B \hookrightarrow {\rm PE}(B)$ is a pure-injective envelope, $B^{++}$ is pure-injective and $B \hookrightarrow B^{++}$ is a pure embedding, one can see that there exists a split monomorphism ${\rm PE}(B) \hookrightarrow B^{++}$ such that the following triangle commutes:
\[
\begin{tikzpicture}[description/.style={fill=white,inner sep=2pt}] 
\matrix (m) [matrix of math nodes, row sep=2.5em, column sep=2.5em, text height=1.25ex, text depth=0.25ex] 
{ 
B & {\rm PE}(B) \\
{} & B^{++} \\
}; 
\path[right hook->] 
(m-1-1) edge (m-1-2) edge (m-2-2)
;
\path[dotted,right hook->]
(m-1-2) edge node[right] {\footnotesize$\exists$} (m-2-2)
;
\end{tikzpicture} 
\]
Then, ${\rm PE}(B)$ is a direct summand of $B^{++}$. On the other hand, $B^{++} \in \mathcal{B}$ since $\mathcal{B}$ is a definable class of modules (see \cite[Corollary 4.6]{MehdiPrest}). Hence, ${\rm PE}(B) \in \mathcal{B}$. 

\item[(iii)] ${\rm Prod}(\mathcal{A}^+) = {\rm PInj}(\mathcal{B})$: First, let us check the containment ${\rm Prod}(\mathcal{A}^+) \subseteq {\rm PInj}(\mathcal{B})$. Let $M \in {\rm Prod}(\mathcal{A}^+)$, that is, $M \simeq \prod_{i \in I} A^+_i$ with $A_i \in \mathcal{A}$ for every $i \in I$. Note that $\prod_{i \in I} A^+_i \simeq (\bigoplus_{i \in I} A_i)^+$ where $\bigoplus_{i \in I} A_i \in \mathcal{A}$ since $\mathcal{A}$ is definable. Then, $M \simeq (\bigoplus_{i \in I} A_i)^+ \in \mathcal{B}$, and $M$ is pure-injective since every character module is pure-injective by \cite[Proposition 5.3.7]{EJ}. Now for the containment ${\rm Prod}(\mathcal{A}^+) \supseteq {\rm PInj}(\mathcal{B})$, let $B$ be a pure-injective module in $\mathcal{B}$. Then, $B$ is a direct summand of $B^{++} = (B^+)^+$ where $B^+ \in \mathcal{A}$, and hence $B \in {\rm Prod}(\mathcal{A}^+)$.
\end{itemize}

\item The results in \cite[Section 4]{MehdiPrest} can be proved in the context of chain complexes using the techniques from the proof of Theorem~\ref{theo:definable_duality_pair} and remark (1) above, along with Lemma~\ref{lem:Prest}.
\end{enumerate}
\end{remark}


\section{\textbf{Relative Gorenstein flat modules and chain complexes}}\label{sec:modules}

We begin this section presenting the notion of \emph{Gorenstein flat objects} relative to a class $\mathcal{B} \subseteq \Mod(R\op)$ of modules or a class $\mathscr{B} \subseteq \Ch(R\op)$ of complexes, and show several properties associated to them. Gorenstein flat modules and complexes will be particular instances, and so, we shall be interested in finding some conditions for $\mathcal{B}$ and $\mathscr{B}$ under which these new relative Gorenstein flat objects are closed under extensions. Moreover, we shall show that it is possible to construct hereditary complete cotorsion pairs from them, as well as left and right approximations. Several of the results presented in this section will be stated and proved only in the setting of $R$-modules, and their chain complex counterparts follow in a very similar way, and so their proofs will be omitted, but there will be some cases where it will be necessary to distinguish between modules and complexes, and where the arguments presented require a careful treatment.  

From now on, given a right $R$-module $N \in \Mod(R\op)$, let us say that a complex $X \in \Ch(R)$ is \emph{$(N \otimes_R -)$-acyclic} if $N \otimes_R X$ is an exact complex of abelian groups, with 
\[
(N \otimes_R X)_m := N \otimes_R X_m
\] 
and differential maps given by 
\[
N \otimes_R \partial^X_m \colon N \otimes_R X_m \to N \otimes_R X_{m-1}
\] 
for every $m \in \mathbb{Z}$. More generally, given a class $\mathcal{N} \subseteq \Mod(R\op)$ of right $R$-modules, we say that $X$ is \emph{$(\mathcal{N} \otimes_R -)$-acyclic} if it is $N$-acyclic for every $N \in \mathcal{N}$. Similarly, given a class $\mathscr{Y} \subseteq \Ch(R\op)$ of complexes of right $R$-modules, $(\mathscr{Y} \otimes -)$-acyclic and $(\mathscr{Y} \otimes^. -)$-acyclic complexes are defined in the same way as their module counterpart, in terms of the tensors $- \otimes -$ and $- \otimes^. -$. 

In what follows, recall that a complex $F \in \Ch(R)$ is \emph{flat} if the induced functor 
\[
- \otimes F \colon \Ch(R\op) \longrightarrow \Ch(\mathbb{Z})
\]
is exact, or equivalently, if $F$ is exact and each cycle $Z_m(F)$ is a flat module (see Garcia Rozas' \cite[Theorem 4.1.3]{jrgr}). Flatness with respect to $\otimes^.$ was studied by the third author in \cite[Proposition 4.5.2]{PerezBook}. We shall say that a complex $L \in (\Ch(R),\otimes^.)$ is \emph{$\otimes^.$-flat} if the functor $- \otimes^. L$ is exact.

Having these notions of acyclicity  and flatness at hand, we present the following definition.

\begin{definition}\label{def:relGflat}
Let $\mathcal{B} \subseteq \Mod(R\op)$ (resp., $\mathscr{B} \subseteq \Ch(R\op)$) be a class of right $R$-modules (resp., a class of complexes of right $R$-modules). We say that:
\begin{enumerate}
\item A module $M \in \Mod(R)$ is \emph{Gorenstein $\mathcal{B}$-flat} if $M = Z_0(F)$ for some $(\mathcal{B} \otimes_R -)$-acyclic and exact complex $F$ of flat modules. 

\item A complex $X \in \Ch(R)$ is \emph{Gorenstein $\bm{\mathscr{B}}$-flat} if $X = Z_0(F_\bullet)$ for some $(\mathscr{B} \otimes -)$-acyclic and exact complex $F_\bullet  \in \Ch(\Ch(R))$ of flat complexes in $\Ch(R)$. 

\item A complex $X \in \Ch(R)$ is \emph{Gorenstein $\mathscr{B}$-flat under $\bm{\otimes^.}$} if $X = Z_0(F_\bullet)$ for some $(\mathscr{B} \otimes^. -)$-acyclic and exact complex $F_\bullet  \in \Ch(\Ch(R))$ of $\otimes^.$-flat complexes in $\Ch(R)$.
\end{enumerate}
\end{definition}

We shall denote by $\mathcal{GF}_{\mathcal{B}}(R)$, $\mathscr{GF}_{\mathscr{B}}(R)$ and $\mathscr{GF}^{\otimes^.}_{\mathscr{B}}(R)$ the classes of Gorenstein $\mathcal{B}$-flat modules in $\Mod(R)$, Gorenstein $\mathscr{B}$-flat complexes and Gorenstein $\mathscr{B}$-flat complexes under $\otimes^.$ in $\Ch(R)$, respectively. We shall use the term ``\emph{relative Gorenstein flat objects}'' to comprise these three families of modules and complexes.

\begin{example}\label{ex:GFB_cases} \
\begin{enumerate}
\item Gorenstein flat modules are obtained by setting $\mathcal{B} = \mathcal{I}(R\op)$ in the previous definition. In this case, we use the notation $\mathcal{GF}(R)$ for the class $\mathcal{GF}_{\mathcal{I}(R\op)}(R)$.

\item Similarly, if we set $\mathscr{B}$ as the class $\mathscr{I}(R\op)$ of injective complexes of right $R$-modules in the previous definition, then $\mathscr{GF}_{\mathscr{B}}(R)$ coincides with the class $\mathscr{GF}(R)$ of Gorenstein flat complexes (see \cite[Definition 3.1]{gang:12:gorflatGF}, for instance).

\item Recall that a module $M \in \Mod(R)$ is \emph{of type $\text{FP}_\infty$} if there exists an exact sequence
\[
\cdots \to P_1 \to P_0 \to M \to 0
\]
with $P_k$ finitely generated and projective for every $k \geq 0$. Let us denote the class of modules of type $\text{FP}_\infty$ by $\mathcal{FP}_\infty(R)$. 

Setting $\mathcal{B} = \mathcal{AC}(R\op) := (\mathcal{FP}_\infty(R\op))^\perp$ as the class of \emph{absolutely clean} right $R$-modules in the previous definition yields the class $\mathcal{GF}_{\rm AC}(R)$ of \emph{Gorenstein AC-flat modules}. These relative Gorenstein flat modules were defined and studied by D. Bravo and the first and second named authors in \cite{BEI17}. 

Some of the properties valid for Gorenstein flat modules carry over to Gorenstein AC-flat modules. For instance, they form a precovering class over any ring $R$. Moreover, every Gorenstein AC-flat module is a direct summand of a strongly Gorenstein AC-flat module. The converse is also true provided that $R$ is a ring over which $\mathcal{GF}_{\rm AC}(R)$ is closed under extensions. The latter also implies that $\mathcal{GF}_{\rm AC}(R)$ is a covering class (See \cite{BEI17} for details). 

One consequence of the results in the present paper is that Gorenstein AC-flat modules are always closed under extensions (see Example \ref{ex:Gorenstein_AC} (2)), and so the latter two properties hold for any ring $R$.

\item Complexes of type $\text{FP}_\infty$ and absolutely clean complexes are defined as their module analogs, and are studied in detail in \cite{BravoGillespie} by Bravo and Gillespie. If we let $\mathscr{AC}(R\op)$ denote the class of absolutely clean complexes in $\Ch(R\op)$, we shall write the class $\mathscr{GF}_{\mathscr{AC}(R\op)}(R)$ of Gorenstein $\mathscr{AC}(R\op)$-flat complexes in $\Ch(R)$ as $\mathscr{GF}_{\rm AC}(R)$. Complexes in $\mathscr{GF}_{\rm AC}(R)$ shall be referred to as \emph{Gorenstein AC-flat complexes}. Such complexes will be closed under extensions (since $\mathscr{AC}(R\op)$ is a definable class due to \cite[Proposition 2.7]{BravoGillespie}) and form a covering class that it is also the left half of a hereditary complete cotorsion pair. (This will be explained in more detail in Theorem \ref{theo:equivalences_GF}, Proposition \ref{prop:GFB_properties} and Corollary \ref{complete}).
\end{enumerate}
\end{example}

\begin{proposition}\label{prop:relative_GF_is_GF}
If $\mathcal{B} \subseteq \Mod(R\op)$ is a class of right $R$-modules containing the class $\mathcal{I}(R\op)$ of injectives, then any Gorenstein $\mathcal{B}$-flat module is, in particular, Gorenstein flat.\footnote{The containments $\mathscr{GF}_{\mathscr{B}}(R) \subseteq \mathscr{GF}(R)$ and $\mathscr{GF}^{\otimes^.}_{\mathscr{B}}(R) \subseteq \mathscr{GF}(R)$ also hold for the case where $\mathscr{B} \subseteq \Ch(R\op)$ is a class containing the class $\mathscr{I}(R\op)$ of injective complexes.} 
\end{proposition}

\begin{proof}
The proofs concerning Gorenstein $\mathcal{B}$-flat modules and Gorenstein $\mathscr{B}$-flat complexes are immediate. However, this is not the case for Gorenstein $\mathscr{B}$-flat complexes under $\otimes^.$. Let $X \in \mathscr{GF}_{\mathscr{B}}^{\otimes^.}(R)$, that is, $X = Z_0(F_\bullet)$ for some exact complex 
\[
F_\bullet = \cdots \to F_1 \to F_0 \to F_{-1} \to \cdots
\] 
of $\otimes^.$-flat complexes such that $B \otimes^. F_\bullet$ is exact for every $B \in \mathscr{B}$. Now consider an injective right $R$-module $I$. Then, $D^0(I)$ is an injective complex, and so $D^0(I) \in \mathscr{B}$. In particular, 
\[
D^0(I) \otimes^. F_\bullet = \cdots \to D^0(I) \otimes^. F_1 \to D^0(I) \otimes^. F_0 \to D^0(I) \otimes^. F_{-1} \to \cdots
\]
is an exact complex of complexes of abelian groups. Thus, for each $m \in \mathbb{Z}$, we have an exact sequence
\begin{align*}
(D^0(I) \otimes^. F_\bullet)_m & = \cdots \to (D^0(I) \otimes^. F_1)_m \to (D^0(I) \otimes^. F_0)_m \to (D^0(I) \otimes^. F_{-1})_m \to \cdots \\
& = \cdots \to [(I \displaystyle\operatorname*{\otimes}_R F_{0,m}) \oplus (I \displaystyle\operatorname*{\otimes}_R F_{0,m-1})] \to [(I \displaystyle\operatorname*{\otimes}_R F_{-1,m}) \oplus (I \displaystyle\operatorname*{\otimes}_R F_{-1,m-1})] \to \cdots \\
& \simeq \cdots \to [I \displaystyle\operatorname*{\otimes}_R (F_{0,m} \oplus F_{0,m-1})] \to [I \displaystyle\operatorname*{\otimes}_R (F_{-1,m} \oplus F_{-1,m-1})] \to \cdots \\
& \simeq I \displaystyle\operatorname*{\otimes}_R (\cdots \to F_{1,m} \oplus F_{1,m-1} \to F_{0,m} \oplus F_{0,m-1} \to F_{-1,m} \oplus F_{-1,m-1} \to \cdots) \\
& = I \displaystyle\operatorname*{\otimes}_R (F_{\bullet,m} \oplus F_{\bullet,m}[1])
\end{align*}
where $F_{\bullet,m}[1]$ is the $1$st suspension of 
\[
F_{\bullet,m} = \cdots \to F_{1,m} \to F_{0,m} \to F_{-1,m} \to \cdots.
\] 
Note that $F_{\bullet,m} \oplus F_{\bullet,m}[1]$ is an exact and $(\mathcal{I}(R\op) \otimes_R -)$-acyclic complex of flat modules such that $Z_0(F_{\bullet,m} \oplus F_{\bullet,m}[1]) = X_m \oplus X_{m-1}$. Thus, $X_m \oplus X_{m-1}$ is a Gorenstein flat module. Since Gorenstein flat modules are closed under direct summands, we have that $X_m$ is also Gorenstein flat. Hence, we have that $X \in \Ch(\mathcal{GF}(R)) = \mathscr{GF}(R)$ by \cite[Corollary 3.12]{gang:12:gorflatGF}.
\end{proof}

\begin{remark} 
For the purposes of this paper, the main class of relative Gorenstein flat complexes which we shall work with is that of Gorenstein $\mathscr{B}$-flat complexes (under the modified tensor product). One reason behind this is that when $\mathscr{B}$ is the class of injective complexes in $\Ch(R\op)$, then we recover the original concept of Gorenstein flat complexes presented in \cite[Definition 4.4]{EnochsGarcia}, and studied subsequently in other works, like for instance \cite{jrgr,gang:12:gorflatGF}. Although this is not actually a limitation to consider $\otimes^.$ instead in Definition~\ref{def:relGflat}.

However, one problem that arises after replacing $\otimes$ by $\otimes^.$ is that the notion of flatness in $\Ch(R)$ changes. Indeed, speaking in a more general setting, if we are given a category with several monoidal structures on it, then we may have distinct notions of (\emph{geometric}) \emph{flat objects} for each structure. The category $\Ch(R)$ for instance has two well known monoidal structures given by $(\otimes,D^0(R))$ and $(\otimes^.,S^0(R))$ in the case $R$ is a commutative ring (although we do not need $R$ to be commutative in order to define $\otimes$ or $\otimes^.$). For the former structure, recall that a chain complex $F \in (\Ch(R),\otimes,D^0(R))$ is flat if, and only if, $F$ is exact and $Z_m(F)$ is a flat module for every $m \in \mathbb{Z}$. On the other hand, a complex $L \in (\Ch(R),\otimes^.,S^0(R))$ is $\otimes^.$-flat if, and only if, $L$ is a complex (not necessarily exact) of flat modules. 
\end{remark}

From now on, let $\Tor^{\Ch}_i(-,-)$ and $\Tor^._i(-,-)$ denote the derived functors of $- \otimes -$ and $- \otimes^. -$. For the former, it is important to recall the properties mentioned throughout \cite{jrgr}. For the latter, the following are easy to note:
\begin{itemize}
\item $\Tor^{\cdot}_0(-,-) = - \otimes^. -$.

\item $\Tor^{\cdot}_i(-,-)$ commutes with direct limits at each variable. 

\item $\Tor^{\cdot}_i(Y,F) = 0$ for every $i \geq 1$ and $Y \in \Ch(R\op)$ if, and only if, $F$ is $\otimes^.$-flat.

\item $(\Tor^{\cdot}_i(S^m(B),F))_n \cong \Tor^R_i(B,F_{n-m})$ for every $m \in \mathbb{Z}$ and $B \in \Mod(R\op)$. 
\end{itemize}

The following lemma provides a useful characterisation of the relative Gorenstein flat objets involving torsion functors, and can be obtained for the module case after using the arguments in Bennis' \cite[Lemma 2.4]{BennisGF}, once the class of injective modules is replaced with the class $\mathcal{B}$. These arguments can be easily adapted to the setting of chain complexes.

\begin{lemma}\label{characterisations}
The following are equivalent for any $M \in \Mod(R)$ and $\mathcal{B} \subseteq \Mod(R\op)$:
\begin{itemize}
\item[(a)] $M$ is Gorenstein $\mathcal{B}$-flat.

\item[(b)] $\Tor_i ^R(B,M) = 0$ for all $i \ge 1$ and $B \in \mathcal{B}$; and there exists an exact and $(\mathcal{B} \otimes_R -)$-acyclic sequence of modules 
\[
0 \rightarrow M \rightarrow F^0 \rightarrow F^1 \rightarrow \cdots
\] 
where each $F^i$ is flat.

\item[(c)] There exists a short exact sequence of modules 
\[
0 \rightarrow M \rightarrow F \rightarrow G \rightarrow 0
\] 
where $F$ is flat and $G$ is Gorenstein $\mathcal{B}$-flat.\footnote{Similar equivalences hold for Gorenstein $\mathscr{B}$-flat complexes (resp., for Gorenstein $\mathscr{B}$-flat complexes under $\otimes^.$), if we replace $\mathcal{B}$ by $\mathscr{B} \subseteq \Ch(R\op)$, $\Tor_i ^R(-,-)$ by $\Tor_i^{\Ch}(-,-)$ (resp., by $\Tor^._i(-,-)$) and flat modules by flat complexes (resp., by $\otimes^.$-flat complexes).}
\end{itemize}
\end{lemma}

In what follows, we shall prove several properties of relative Gorenstein flat objects. The very first one to show in our list is that the classes $\mathcal{GF}_{\mathcal{B}}(R)$ and $\mathscr{GF}_{\mathscr{B}}(R)$ are closed under extensions, provided that a couple of sufficient conditions are satisfied by $\mathcal{B}$. The \emph{absolute case} for modules, that is $\mathcal{B} = \mathcal{I}(R^{\rm op})$, was settled by \v{S}aroch and \v{S}\v{t}ov\'{\i}\v{c}ek in \cite[Theorem 3.10]{SarochStovicek}, where they show that the class $\mathcal{GF}(R)$ of Gorenstein flat modules can be written as the left orthogonal class 
\begin{align}\label{eqn1}
\mathcal{GF}(R) & = {}^\perp(\mathcal{C}(R) \cap (\mathcal{PGF}(R))^\perp)
\end{align}
for any arbitrary ring $R$. Here, $\mathcal{PGF}(R)$ denotes the class of \emph{projectively coresolved Gorenstein flat modules} \cite[Section 3]{SarochStovicek}. The closure under extensions for $\mathscr{GF}_{\mathscr{B}}(R)$ in the case where $\mathscr{B} = \mathscr{I}(R\op)$ is a direct consequence of \cite[Corollary 3.12.]{gang:12:gorflatGF}. 

There is a long path to go through before showing that $\mathcal{GF}_{\mathcal{B}}(R)$ and $\mathscr{GF}_{\mathscr{B}}(R)$ are closed under extensions. For this goal, the classes $\mathcal{B}$ and $\mathscr{B}$ will be required to satisfy a couple of conditions which are related to the notion of definable classes.


\subsection*{\textbf{Relative projectively coresolved Gorenstein flat modules and complexes}}

In order to show a relative version of \eqref{eqn1} and its chain complex counterpart, we present the following analog of projectively coresolved Gorenstein flat modules (see \cite[Section 3]{SarochStovicek}).

\begin{definition}\label{def:projGF}
Let $\mathcal{B} \subseteq \mathsf{Mod}(R\op)$ be a class of right $R$-modules. We say that a module $M$ is \emph{projectively coresolved Gorenstein $\mathcal{B}$-flat} if $M = Z_0(P)$ for some $(\mathcal{B} \otimes_R -)$-acyclic and exact complex $P$ of projective modules. \emph{Projectively coresolved Gorenstein $\mathscr{B}$-flat complexes} are defined in the same way for any class $\mathscr{B} \subseteq \Ch(R\op)$ and regarding the tensor product $- \otimes -$. 
\end{definition}

In what follows, let us denote by $\mathcal{PGF}_{\mathcal{B}}(R)$ and $\mathscr{PGF}_{\mathscr{B}}(R)$ the classes of projectively coresolved Gorenstein $\mathcal{B}$-flat modules and Gorenstein $\mathscr{B}$-flat complexes, respectively.  

The purpose of this section is to find some sufficient conditions for $\mathcal{B} \subseteq \mathsf{Mod}(R\op)$ and $\mathscr{B} \subseteq \Ch(R\op)$ so that $\mathcal{PGF}_{\mathcal{B}}(R)$ and $\mathscr{PGF}_{\mathscr{B}}(R)$ are the left halves of complete cotorsion pairs in $\Mod(R)$ and $\Ch(R)$, respectively. Namely, we shall need the conditions specified in the following definition.

\begin{definition}
We say that a subcategory $\mathcal{B}$ of a finitely accessible additive category with products is \emph{semi-definable} if it is closed under products and contains an elementary cogenerator of its definable closure.
\end{definition}

For any semi-definable class $\mathcal{B} \subseteq \mathsf{Mod}(R\op)$, we shall be able to show the equality
\begin{align}\label{eqn:GF_description}
\mathcal{GF}_{\mathcal{B}}(R) = {}^\perp(\mathcal{C}(R) \cap (\mathcal{PGF}_{\mathcal{B}}(R))^\perp).
\end{align}

The following result is the relative version of \cite[Lemma 3.1]{SarochStovicek}. Its proof is very similar to the proof of the absolute case (that is, setting $\mathcal{B} = \mathcal{I}(R\op)$), and an overview for this can be seen in \cite[Section 3]{EstradaFuIacob}.

\begin{lemma}\label{lem:Tor}
Let $\mathcal{B} \subseteq \mathsf{Mod}(R\op)$ be any class of right $R$-modules. If $N$ is a projectively coresolved Gorenstein $\mathcal{B}$-flat module, then $N$ is a direct summand of a module $M$ such that $M \simeq P / M$ with $P$ projective and $\Tor^R_i(B,M) = 0$ for every $B \in \mathcal{B}$ and $i \geq 1$.\footnote{A similar statement is also true for projectively coresolved Gorenstein $\mathscr{B}$-flat complexes and torsion functors $\Tor^{\Ch}_i(-,-)$.} 
\end{lemma}

We have the following relations between the classes $\mathcal{PGF}_{\mathcal{B}}(R)$, $\mathcal{F}(R)$ and the definable closure $\langle R \rangle$ of the ground ring $R$, which are similar to those appearing in \cite{SarochStovicek} for the absolute case. We only give some comments on the proof. In what follows, we shall denote by $\mathscr{F}(R)$ the class of flat complexes in $\Ch(R)$.

\begin{proposition}\label{prop:PGFB_definable_R}
Let $\mathcal{B} \subseteq \Mod(R\op)$ be a class containing the injective right $R$-modules. The following containments hold true for any ring $R$:
\begin{enumerate}
\item $\mathcal{PGF}_{\mathcal{B}}(R) \subseteq {}^\perp\langle R \rangle \subseteq {}^{\perp}\mathcal{F}(R)$.

\item $\mathcal{F}(R) \subseteq \langle R \rangle \subseteq (\mathcal{PGF}_{\mathcal{B}}(R))^\perp$.
\end{enumerate}
Similarly, the following containments hold true for any ring $R$ and any class $\mathscr{B} \subseteq \Ch(R\op)$ containing the injective complexes of right $R$-modules:
\begin{enumerate}
\setcounter{enumi}{2}
\item $\mathscr{PGF}_{\mathscr{B}}(R) \subseteq {}^\perp\langle \bigoplus_{m \in \mathbb{Z}} D^m(R) \rangle \subseteq {}^{\perp}\mathscr{F}(R)$.

\item $\mathscr{F}(R) \subseteq \langle \bigoplus_{m \in \mathbb{Z}} D^m(R) \rangle \subseteq (\mathscr{PGF}_{\mathscr{B}}(R))^\perp$.
\end{enumerate}
In particular, projectively coresolved Gorenstein $\mathcal{B}$-flat modules and projectively coresolved Gorenstein $\mathscr{B}$-flat complexes are Gorenstein projective.
\end{proposition}

\begin{proof} \
\begin{itemize}
\item \underline{Module setting}: Note that any definable class is closed under direct summands and under coproducts. So in particular, $\langle R \rangle$ contains the class of projective modules. Moreover, as every flat module is the direct limit of a directed family of projective modules, we have that the class $\mathcal{F}(R)$ of flat modules is contained in $\langle R \rangle$. Using these observations and proceeding in a similar way as in \cite[Theorem 3.4]{SarochStovicek}, we obtain the relations (1) and (2)  after applying Lemma \ref{lem:Tor} and \cite[Proposition 3.2]{SarochStovicek}. 

\item \underline{Chain complex setting}: Relations (3) and (4) follow is a similar way. For instance, in the proof of (3), one needs to take the complex $I$ defined as $\underline{\Hom}(\bigoplus_{m \in \mathbb{Z}} D^m(R), D^1(\mathbb{Q / Z}))$, and use the facts that every complex is a pure subcomplex of its double dual \cite[Part 4 of Proposition 5.1.4]{jrgr}, and that Gorenstein projective complexes are closed under direct summands \cite[Theorem 2.3]{yang:11:gorflat}. For the proof of (4), on the other hand, one just needs to notice that Lazard's Theorem also holds for chain complexes \cite[Theorem 4.1.3]{jrgr}.
\end{itemize} 
\end{proof}

The following result is the relative version of the absolute case proven in \cite[Lemma 3.7]{SarochStovicek}. 

Given a subcategory $\mathcal{S}$ of a Grothendieck category $\mathcal{G}$, recall that an object $M \in \mathcal{G}$ is a \emph{transfinite extension of $\mathcal{S}$} (or an \emph{$\mathcal{S}$-filtration}) if $M \simeq \varinjlim_{\alpha < \lambda} S_\alpha$, for some ordinal $\lambda > 0$, and such that:
\begin{itemize}
\item For every $\alpha + 1 < \lambda$, the morphism $S_\alpha \to S_{\alpha + 1}$ is a monomorphism.

\item $S_0 = 0$ and $S_{\alpha + 1} / S_\alpha \in \mathcal{S}$ for every $\alpha + 1 < \lambda$.
\end{itemize} 
In this sense, $\mathcal{S}$ is said to be \emph{closed under transfinite extensions} if every $\mathcal{S}$-filtered object of $\mathcal{G}$ belongs to $\mathcal{S}$.

\begin{theorem}\label{theo:closure_PGFB}
Let $R$ be an arbitrary ring. Then, the class $\mathcal{PGF}_{\mathcal{B}}(R)$ of projectively coresolved Gorenstein $\mathcal{B}$-flat modules is resolving and closed under transfinite extensions.\footnote{This result is also valid if we replace $\mathcal{B}$ by a class $\mathscr{B} \subseteq \Ch(R\op)$ and the class $\mathcal{PGF}_{\mathcal{B}}(R)$ by the class $\mathscr{PGF}_{\mathscr{B}}(R)$ of projectively coresolved Gorenstein $\mathscr{B}$-flat complexes.}
\end{theorem}

\begin{proof} 
Let us split the proof into two parts:
\begin{enumerate}
\item It is clear that $\mathcal{PGF}_{\mathcal{B}}(R)$ contains the class $\mathcal{P}(R)$ of projective modules. We only prove that $\mathcal{PGF}_{\mathcal{B}}(R)$ is closed under extensions, as this will imply that $\mathcal{PGF}_{\mathcal{B}}(R)$ is also closed under taking kernels of epimorphisms between modules in $\mathcal{PGF}_{\mathcal{B}}(R)$, by an argument similar to \cite[Lemma 3.7]{SarochStovicek}. Thus, consider a short exact sequence
\[
\varepsilon \colon 0 \to M_1 \xrightarrow{\alpha} M_2 \xrightarrow{\beta} M_3 \to 0
\]
with $M_1, M_3 \in \mathcal{PGF}_{\mathcal{B}}(R)$. By Definition~\ref{def:projGF}, we can consider the following short exact sequences
\begin{align*}
\varepsilon_1 \colon & 0 \to M_1 \xrightarrow{g} P_1 \to M'_1 \to 0, \\
\varepsilon_3 \colon & 0 \to M_3 \xrightarrow{h} P_3 \to M'_3 \to 0,
\end{align*}
where $P_1$ and $P_3$ are projective, and $M'_1, M'_3 \in \mathcal{PGF}_{\mathcal{B}}(R)$. Since $\Ext^1_R(M_3,P_1) = 0$ by Proposition \ref{prop:PGFB_definable_R}, we have that $\Hom_R(\varepsilon,P_1)$ is exact. It follows that there exists a morphism $m_1 \colon M_2 \to P_1$ such that $m_1 \circ \alpha = g$. By the universal property of coproducts and cokernels, we have the following commutative diagram with exact rows and columns:
\begin{equation}\label{fig1} 
\parbox{2.5in}{
\begin{tikzpicture}[description/.style={fill=white,inner sep=2pt}] 
\matrix (m) [ampersand replacement=\&, matrix of math nodes, row sep=2.5em, column sep=2.5em, text height=1.25ex, text depth=0.25ex] 
{ 
0 \& M_1 \& M_2 \& M_3 \& 0 \\
0 \& P_1 \& P_1 \oplus P_3 \& P_3 \& 0 \\
0 \& M'_1 \& M'_2 \& M'_3 \& 0 \\
}; 
\path[->] 
(m-1-1) edge (m-1-2) (m-1-2) edge node[above] {\footnotesize$\alpha$} (m-1-3) (m-1-3) edge node[above] {\footnotesize$\beta$} (m-1-4) (m-1-4) edge (m-1-5)
(m-2-1) edge (m-2-2) (m-2-2) edge (m-2-3) (m-2-3) edge (m-2-4) (m-2-4) edge (m-2-5)
(m-3-1) edge (m-3-2) (m-3-2) edge (m-3-3) (m-3-3) edge (m-3-4) (m-3-4) edge (m-3-5)
(m-1-3) edge node[above,sloped] {\footnotesize$m_1$} (m-2-2)
;
\path[>->]
(m-1-2) edge node[right] {\footnotesize$g$} (m-2-2) (m-1-3) edge node[right] {\footnotesize$m$} (m-2-3) (m-1-4) edge node[right] {\footnotesize$h$} (m-2-4)
;
\path[->>]
(m-2-2) edge (m-3-2) (m-2-3) edge (m-3-3) (m-2-4) edge (m-3-4) 
; 
\end{tikzpicture} 
}
\end{equation} 
Repeating the same argument infinitely many times, we obtain a long exact sequence 
\[
0 \to M_2 \xrightarrow{m} P_1 \oplus P_3 \to P'_1 \oplus P'_3 \to \cdots
\]
of projective modules with cycles in ${\rm Ker}(\Tor^R_i(\mathcal{B},-))$ for every $i \geq 1$. Indeed, for the latter condition note that for each $i \geq 1$ we have an exact sequence
\[
\Tor^R_i(B,M_1) \to \Tor^R_i(B,M_2) \to \Tor^R_i(B,M_3)
\]
where $\Tor^R_i(B,M_1) = 0$ and $\Tor^R_i(B,M_3) = 0$ for every $B \in \mathcal{B}$, by Lemmas~\ref{characterisations} and \ref{lem:Tor}. Hence, $\Tor^R_i(B,M_2) = 0$ for every $B \in \mathcal{B}$ and $i \geq 1$. Similarly, one can note that $\Tor^R_i(B,M'_2) = 0$ and so on. Therefore, $M_2 \in \mathcal{PGF}_{\mathcal{B}}(R)$ by Lemma~\ref{lem:Tor} again. 

\item To show that $\mathcal{PGF}_{\mathcal{B}}(R)$ is closed under transfinite extensions, suppose that we are given a module $M$ written as $M = \varinjlim_{\alpha < \lambda} M_\alpha$ for some ordinal $\lambda > 0$ such that $M_0 \in \mathcal{PGF}_{\mathcal{B}}(R)$ and $M_{\alpha + 1} / M_\alpha \in \mathcal{PGF}_{\mathcal{B}}(R)$ for every $\alpha + 1 < \lambda$. Proceeding by transfinite induction, suppose we have constructed a projective coresolution 
\[
\rho_\alpha \colon 0 \to M_\alpha \to P^0_{\alpha} \to P^1_{\alpha} \to \cdots
\]
with cycles in ${\rm Ker}(\Tor^R_i(B,-))$ for every $B \in \mathcal{B}$ and $i \geq 1$. By assumption, we also have a projective coresolution 
\[
0 \to M_{\alpha + 1} / M_\alpha \to P^0_{\alpha,\alpha+1} \to P^1_{\alpha,\alpha+1} \to \cdots
\]
satisfying the same condition on its cycles. Proceeding as in part (1) above, we have the following commutative diagram with exact rows and columns:
\begin{equation}\label{fig3} 
\parbox{4in}{
\begin{tikzpicture}[description/.style={fill=white,inner sep=2pt}] 
\matrix (m) [ampersand replacement=\&, matrix of math nodes, row sep=2.5em, column sep=2.5em, text height=1.25ex, text depth=0.25ex] 
{ 
0 \& M_\alpha \& M_{\alpha + 1} \& M_{\alpha + 1} / M_\alpha \& 0 \\
0 \& P^0_\alpha \& P^0_\alpha \oplus P^0_{\alpha,\alpha+1} \& P^0_{\alpha,\alpha+1} \& 0 \\
0 \& P^1_\alpha \& P^1_\alpha \oplus P^1_{\alpha,\alpha+1} \& P^1_{\alpha,\alpha+1} \& 0 \\
{} \& \vdots \& \vdots \& \vdots \& {} \\
}; 
\path[->] 
(m-1-1) edge (m-1-2) (m-1-2) edge (m-1-3) (m-1-3) edge (m-1-4) (m-1-4) edge (m-1-5)
(m-2-1) edge (m-2-2) (m-2-2) edge (m-2-3) (m-2-3) edge (m-2-4) (m-2-4) edge (m-2-5)
(m-3-1) edge (m-3-2) (m-3-2) edge (m-3-3) (m-3-3) edge (m-3-4) (m-3-4) edge (m-3-5)
(m-2-2) edge (m-3-2) (m-2-3) edge (m-3-3) (m-2-4) edge (m-3-4)
(m-3-2) edge (m-4-2) (m-3-3) edge (m-4-3) (m-3-4) edge (m-4-4)
; 
\path[>->]
(m-1-2) edge (m-2-2) (m-1-3) edge (m-2-3) (m-1-4) edge (m-2-4)
;
\end{tikzpicture} 
}
\end{equation} 
where the central column is a projective coresolution with cycles (including $M_{\alpha + 1}$ itself) in ${\rm Ker}(\Tor^R_i(B,-))$ for every $B \in \mathcal{B}$ and $i \geq 1$. Then, we can set 
\begin{align*}
P^i_{\alpha + 1} & := P^i_{\alpha} \oplus P^i_{\alpha, \alpha + 1} & \text{for every $\alpha + 1 < \lambda$}, \\
P^i_\beta & := \varinjlim_{\alpha < \beta} P^i_\alpha & \text{for every limit ordinal $\beta \leq \lambda$}. 
\end{align*} 
Now fix $N \in \Mod(R)$ and let $\alpha + 1 < \lambda$. We have that 
\[
\Ext^1_R(P^i_0,N) = 0 \text{ \ and \ } \Ext^1_R(P^i_{\alpha + 1} / P^i_\alpha,N) = 0.
\] 
By Eklof's Lemma (see for instance \cite[Theorem 7.3.4]{EJ}), we get $\Ext^1_R(P^i_\lambda,N) = 0$ for $P^i_\lambda := \varinjlim_{\alpha < \lambda} P^i_\alpha$. Since $N$ is arbitrary, we have that each $P^i_\lambda$ is projective. Then, taking the direct limit of all the $\rho_\alpha$ yields a projective coresolution of $M$, say $\rho$. It is only left to show that the cycles of $\rho$ belong to ${\rm Ker}(\Tor^R_i(B,-))$ for every $B \in \mathcal{B}$ and $i \geq 1$. For each $\alpha + 1 < \lambda$, we know that $M_\alpha \in {\rm Ker}(\Tor^R_i(B,-))$ for every $B \in \mathcal{B}$ and $i \geq 1$. Furthermore, for every limit ordinal $\beta \leq \lambda$, we have $M_\beta \in {\rm Ker}(\Tor^R_i(B,-))$ since $\Tor^R_i(B,-)$ preserves direct limits. Therefore, the previous implies that $\mathcal{PGF}_{\mathcal{B}}(R)$ is closed under transfinite extensions.
\end{enumerate}

The corresponding result for $\mathscr{PGF}_{\mathscr{B}}(R)$ follows similarly using the chain complex version of Lemma \ref{lem:Tor}, along with Eklof's Lemma (which is also valid in $\Ch(R)$).
\end{proof}

We now focus on proving that $\mathcal{PGF}_{\mathcal{B}}(R)$ is the left half of a hereditary complete cotorsion pair in $\mathsf{Mod}(R)$. This is a key result for showing that $\mathcal{GF}_{\mathcal{B}}(R)$ is closed under extensions.

\begin{lemma}\label{lem:characterisation_pcoresolved}
Let $\mathcal{B} \subseteq \mathsf{Mod}(R^{\rm op})$ be a class closed under products, and $P$ be an exact complex of projective modules. Consider the associated duality pair $(\mathcal{A},\langle \mathcal{B} \rangle)$ from Theorem \ref{theo:definable_duality_pair}. The following conditions are equivalent:
\begin{itemize}
\item[(a)] $B \otimes_R P$ is exact for every $B \in \mathcal{B}$.

\item[(b)] $N \otimes_R P$ is exact for every $N \in \langle \mathcal{B} \rangle$.

\item[(c)] $\mathsf{Hom}_R(P,A)$ is exact for every $A \in \mathcal{A}$.
\end{itemize}
Moreover, in the case where $B_0$ is an elementary cogenerator of $\langle \mathcal{B} \rangle$, then the previous are also equivalent to:
\begin{itemize}
\item[(d)] $N \otimes_R P$ is exact for every $N \in \langle B_0 \rangle$.

\item[(e)] $\mathsf{Hom}_R(P,A)$ is exact for every $A \in \langle B^+_0 \rangle$. 
\end{itemize} 

In the setting of chain complexes, let $\mathscr{B} \subseteq \Ch(R\op)$ be a semi-definable class of complexes of right $R$-modules, and consider the associated duality pair $(\mathscr{A},\langle \mathscr{B} \rangle)$ from Theorem \ref{theo:definable_duality_pair}, where $\langle \mathscr{B} \rangle$ has an elementary cogenerator $B_0$. Then, the following conditions are equivalent for every exact complex $P_\bullet$ of projective complexes in $\Ch(R)$:
\begin{itemize}
\item[(i)] $B \otimes P_\bullet$ is exact for every $B \in \mathscr{B}$.

\item[(ii)] $Y \otimes P_\bullet$ is exact for every $Y \in \langle \mathscr{B} \rangle$. 

\item[(iii)] $Y \otimes P_\bullet$ is exact for every $Y \in {\rm CoGen}_\ast(B_0)$. 

\item[(iv)] $\Hom_{\Ch}(P_\bullet,X)$ is exact for every $X \in \langle B^+_0 \rangle$. 
\end{itemize}
\end{lemma}

\begin{proof} 
We only prove the equivalence between (a), (b) and (c). The corresponding assertion regarding (d) and (e) follows by Theorem \ref{theo:definable_duality_pair}. Since $(\mathcal{A},\langle \mathcal{B} \rangle)$ is a duality pair and $\langle \mathcal{B} \rangle$ is closed under pure epimorphic images, we have by \cite[Theorem A.6]{BGH} that (b) and (c) are equivalent. On the other hand, (b) $\Rightarrow$ (a) is trivial. So it suffices to show (a) $\Rightarrow$ (b).

Suppose the complex $P$ satisfies $B \otimes_R P$ is exact for every $B \in \mathcal{B}$. Let $N \in \langle \mathcal{B} \rangle$. Then $N$ can be regarded as a pure submodule (or a pure epimorphic image) of $N'$, where $N'$ is a direct limit or a direct product of elements in $\mathcal{B}$ (see for instance \cite{PrestPurity}). For the latter case, we have that $N' \in \mathcal{B}$ since $\mathcal{B}$ is closed under direct products, and so $N' \otimes_R P$ is an exact complex. On the other hand, in the former case let us write $N' = \varinjlim_I B_i$ for some directed set $I$, with $B_i \in \mathcal{B}$. Since tensor products preserve direct limits, and any direct limit of exact complexes is exact, we have that 
\[
N' \otimes_R P \cong \varinjlim_{I} (B_i \otimes_R P)
\] 
is an exact complex. In any case, we have that $N \otimes_R P$ is a subcomplex of the exact complex $N' \otimes_R P$. In what remains, let us show that $N \otimes_R P$ has to be exact. Since we have a pure embedding $N \hookrightarrow N'$, it follows that $N^+$ is a direct summand of $(N')^+$. This in turn implies that $\Hom_R(P,N^+)$ is a direct summand of $\Hom_R(P,(N')^+)$. Note also that the complex $\Hom_R(P,(N')^+)$ is exact since 
\[
\Hom_R(P,(N')^+) \cong \Hom_{\mathbb{Z}}(N' \otimes_R P, \mathbb{Q / Z}),
\] 
$N' \otimes_R P$ is exact and $\mathbb{Q / Z}$ is an injective $\mathbb{Z}$-module. Hence, the complex $\Hom_R(P,N^+)$ is exact since exact complexes are closed under direct summands. Using again the isomorphism 
\[
\Hom_R(P,N^+) \cong \Hom_{\mathbb{Z}}(N \otimes_R P, \mathbb{Q / Z})
\] 
and the fact that $\mathbb{Q / Z}$ is an injective cogenerator in $\mathsf{Ab}$, we finally have that the complex $N \otimes_R P$ is exact. 

The statement concerning chain complexes follows by a similar argument and by using the chain complex version of \cite[Theorem A.6]{BGH} proved by Gillespie in \cite[Theorem 5.9]{GillespieDingComplexes}.
\end{proof}

\begin{remark}
Let $\mathcal{B} = \mathcal{I}(R^{\rm op})$. In this case, we know $\mathcal{I}(R^{\rm op})$ is closed under products. Moreover, we can show that $\langle \mathcal{I}(R^{\rm op}) \rangle = \langle R^+ \rangle$, and thus Lemma~\ref{lem:characterisation_pcoresolved} coincides with \v{S}aroch and \v{S}\v{t}ov\'{\i}\v{c}ek's \cite[Corollary 3.5]{SarochStovicek}.

In order to show the equality $\langle \mathcal{I}(R^{\rm op}) \rangle = \langle R^+ \rangle$, note first that since $R^+$ is injective, we have $\langle R^+ \rangle \subseteq \langle \mathcal{I}(R^{\rm op}) \rangle$. For the remaining inclusion, it suffices to show $\mathcal{I}(R^{\rm op}) \subseteq \langle R^+ \rangle$. Let $E$ be an injective right $R$-module. Consider an epimorphism $R^{(I)} \twoheadrightarrow E^+$ for some index set $I$. Since the functor $\mathsf{Hom}_{\mathbb{Z}}(-,\mathbb{Q / Z})$ is exact, we have a monomorphism 
\[
E^{++} \rightarrowtail (R^{(I)})^+ = (R^+)^I.
\] 
On the other hand, we have a (pure) monomorphism $E \hookrightarrow E^{++}$. It follows that $E$ is a pure submodule of $(R^+)^I$, where $(R^+)^I \in \langle R^+ \rangle$. Hence, $E \in \langle R^+ \rangle$. 

Let $(\mathcal{A},\langle \mathcal{I}(R^{\rm op}) \rangle)$ be the corresponding duality pair from Theorem \ref{theo:definable_duality_pair}. We show that $\mathcal{A} = \langle \mathcal{P}(R) \rangle = \langle R \rangle$. By Theorem \ref{theo:definable_duality_pair}, we have that $\mathcal{A} = \langle R^{++} \rangle$. Moreover, $R$ is a pure submodule of $R^{++}$, so the inclusion $\mathcal{A} \supseteq \langle R \rangle$ follows. On the other hand, the containment $\mathcal{A} \subseteq \langle R \rangle$ follows since $R^{++} \in \langle R \rangle$ by Lemma~\ref{lem:Prest}. Note also that $\langle \mathcal{P}(R) \rangle \supseteq \langle R \rangle$ and $\mathcal{P}(R) \subseteq \langle R \rangle$ are clear, since every projective module is a direct summand (and so a pure submodule) of a free module. 

Therefore, we have a duality pair of the form
\[
\mathfrak{P} = (\langle \mathcal{P}(R) \rangle, \langle \mathcal{I}(R^{\rm op}) \rangle) = (\langle R \rangle,\langle R^+ \rangle).
\]
\end{remark}

The importance of semi-definable classes has to do with the fact (to be proved) that it is possible to construct complete cotorsion pairs from (projectively coresolved) Gorenstein flat modules which are relative with respect to such classes.  In this sense, we have the following relativisation of \cite[Theorem 3.8]{SarochStovicek}.

\begin{theorem}\label{theo:cotorsion_pair_PGFB}
Let $\mathcal{B} \subseteq \mathsf{Mod}(R^{\rm op})$ be a semi-definable class. Then, 
\[
(\mathcal{PGF}_{\mathcal{B}}(R),(\mathcal{PGF}_{\mathcal{B}}(R))^\perp)
\] 
is a hereditary complete cotorsion pair in $\Mod(R)$ cogenerated by a set.\footnote{A similar result holds in the category $\Ch(R)$ of chain complexes after replacing $\mathcal{PGF}_{\mathcal{B}}(R)$ by $\mathscr{PGF}_{\mathscr{B}}(R)$, where $\mathscr{B} \subseteq \Ch(R^{\rm op})$ is semi-definable.} 
\end{theorem}

\begin{proof} Although the proofs for modules and chain complexes are similar, it will be important to point out some remarks concerning the latter setting. 
\begin{itemize}
\item \underline{Module setting}: First, let us note by G\"obel and Trlifaj's \cite[Definition 4.1.9 and Lemma 4.1.10]{GT} that it is possible to find a regular cardinal $\nu$ such that $R$ is a $\nu$-Noetherian ring, that is, each right ideal of $R$ is $\leq \nu$-generated. Thus, let us consider the set $\mathcal{S}_{\mathcal{B}}$ of representatives of $\nu$-presented modules in $\mathcal{PGF}_{\mathcal{B}}(R)$. Let $(\mathcal{F}_{\mathcal{B}},\mathcal{G}_{\mathcal{B}})$ denote the cotorsion pair cogenerated by $\mathcal{S}_{\mathcal{B}}$, that is, $\mathcal{G}_{\mathcal{B}} = \mathcal{S}_{\mathcal{B}}^\perp$. In what follows, we show that this cotorsion pair coincides with $(\mathcal{PGF}_{\mathcal{B}}(R),(\mathcal{PGF}_{\mathcal{B}}(R))^\perp)$. For, it suffices to show the equality $\mathcal{PGF}_{\mathcal{B}}(R) = \mathcal{F}_{\mathcal{B}}$. Notice that $R \in \mathcal{S}_{\mathcal{B}}$, and so by \cite[Corollary 3.2.4]{GT} the class $\mathcal{F}_{\mathcal{B}}$ consists of all direct summands of $\mathcal{S}_{\mathcal{B}}$-filtered modules. 
\begin{itemize}
\item $\mathcal{PGF}_{\mathcal{B}}(R) \subseteq \mathcal{F}_{\mathcal{B}}$: Let $N \in \mathcal{PGF}_{\mathcal{B}}(R)$. By Lemma~\ref{lem:Tor}, $N$ is a direct summand of a module $M \in \mathsf{Mod}(R)$ satisfying $M \simeq P / M$ with $P$ projective, and such that $\mathsf{Tor}^R_i(B,M) = 0$ for every $B \in \mathcal{B}$ and $i > 0$. We show that $M$ is $\mathcal{S}_{\mathcal{B}}$-filtered. We are given a short exact sequence
\[
\varepsilon \colon 0 \to M \xrightarrow{f} P \to M \to 0.
\]
Let $B_0 \in \mathcal{B}$ be an elementary cogenerator of $\left< \mathcal{B} \right>$. From the conditions we are assuming for $M$ and $\mathcal{B}$, we can note that $B^{\theta}_0 \otimes_R f$ is monic for every cardinal $\theta$ since $B^{\theta}_0 \in \mathcal{B}$. Then, setting $I := B_0$ in \cite[Proposition 3.2]{SarochStovicek}, we have that $M \in {}^\perp\langle B^+_0 \rangle$. Now setting $\mathcal{D} := \mathsf{Mod}(R)$ in \cite[Proposition 3.6]{SarochStovicek}, we have that there exists a filtration 
\[
\mathfrak{F}_{\mathcal{B}} = \{ \varepsilon_\alpha \colon 0 \to M_\alpha \xrightarrow{f_\alpha} P_\alpha \to M_\alpha \to 0 \mbox{ : } \alpha \leq \sigma \}
\]
of $\varepsilon$ such that $P_{\alpha + 1} / P_\alpha$ is projective and $M_{\alpha + 1} / M_\alpha$ is $\nu$-presented and belongs to ${}^\perp\langle B^+_0 \rangle$. Note that for each $\alpha < \sigma$ we have an exact complex $Q_\alpha$ of projectives with $Z_m(Q_\alpha) = M_\alpha$ for every $m \in \mathbb{Z}$. Moreover, the quotient $Q_{\alpha + 1} / Q_\alpha$ is an exact complex of projective modules such that $Z_m(Q_{\alpha + 1} / Q_\alpha) = M_{\alpha + 1} / M_\alpha \in {}^\perp\langle B^+_0 \rangle$ for every $m \in \mathbb{Z}$, that is, $\mathsf{Hom}_R(Q_{\alpha + 1} / Q_\alpha,A)$ is exact for every $A \in \langle B^+_0 \rangle$. By Lemma~\ref{lem:characterisation_pcoresolved}, we have that $M_{\alpha + 1} / M_\alpha \in \mathcal{PGF}_{\mathcal{B}}(R)$ for every $\alpha < \sigma$. It follows that $M$ is $\mathcal{S}_{\mathcal{B}}$-filtered, and hence $N \in \mathcal{F}_{\mathcal{B}}$. 

\item $\mathcal{PGF}_{\mathcal{B}}(R) \supseteq \mathcal{F}_{\mathcal{B}}$: Let $N \in \mathcal{F}_{\mathcal{B}}$. Then, $N$ is a direct summand of a $\mathcal{S}_{\mathcal{B}}$-filtered module $M$, that is, $M = \varinjlim_{\alpha \leq \sigma} M_\alpha$ such that $M_0 = 0$ and $M_{\alpha + 1} / M_\alpha \in \mathcal{S}_{\mathcal{B}}$ for every $\alpha < \sigma$. In other words, $M$ is a transfinite extension of $\mathcal{PGF}_{\mathcal{B}}(R)$. By Theorem \ref{theo:closure_PGFB}, we have that $M \in \mathcal{PGF}_{\mathcal{B}}(R)$. It follows that $N \in \mathcal{PGF}_{\mathcal{B}}(R)$. 
\end{itemize}

So far, we have proved that $(\mathcal{PGF}_{\mathcal{B}}(R),(\mathcal{PGF}_{\mathcal{B}}(R))^\perp)$ is a cotorsion pair cogenerated by the set $\mathcal{S}_{\mathcal{B}}$, and so it is complete. The fact that this pair is hereditary follows by Theorem~\ref{theo:closure_PGFB}.

\item \underline{Chain complex setting}: The proof in the category of chain complexes follows similarly. Indeed, we can consider the set $\mathscr{S}_{\mathscr{B}}$ of representatives of $\nu$-presented complexes in $\mathscr{PGF}_{\mathscr{B}}(R)$, for some regular cardinal $\nu$ making $R$ a $\nu$-Noetherian ring. By \cite[Proposition 1.7]{Stovicek}, ${}^\perp(\mathscr{S}_{\mathscr{B}}^\perp)$ coincides with the class of direct summands of $\mathscr{S}_{\mathscr{B}}$-filtered complexes, since $\mathscr{S}_{\mathscr{B}}$ contains the family of generators of $\Ch(R)$ given by $D^m(R)$ with $m \in \mathbb{Z}$. The proof of the inclusion ${}^\perp(\mathscr{S}_{\mathscr{B}}^\perp) \supseteq \mathscr{PGF}_{\mathscr{B}}(R)$ follows as in the module case, by the second part of Lemma \ref{lem:characterisation_pcoresolved} and the chain complex versions of Lemma \ref{lem:Tor} and \cite[Propositions 3.2 and 3.6]{SarochStovicek}. It is important to mention that the techniques from {\v{S}}aroch and {{\v{S}{t'o}}}v{\'{\i}}{\v{c}}ek's work are also valid in the category $\Ch(R)$ of chain complexes, due to the comments at the beginning of Section 1 in \cite{SarochStovicek}. The rest of the proof follows by the fact that $\mathscr{PGF}_{\mathscr{B}}(R)$ is a resolving class closed under transfinite extensions, due to the chain complex version of Theorem \ref{theo:closure_PGFB}. 
\end{itemize}
\end{proof}


\subsection*{\textbf{Approximations and cotorsion pairs from relative Gorenstein flat objects}}

We are now ready to state the following characterisation of $\mathcal{GF}_{\mathcal{B}}(R)$ involving the class $\mathcal{PGF}_{\mathcal{B}}(R)$, which is a relative version of the absolute case $\mathcal{B} = \mathcal{I}(R^{\rm op})$ proved in \cite[Theorem 3.10]{SarochStovicek}. We shall also extend this relative version of \cite[Theorem 3.10]{SarochStovicek} to the category of complexes.

\begin{theorem}\label{theo:equivalences_GF}
Let $\mathcal{B}$ be a semi-definable class of right $R$-modules. Then, the following conditions are equivalent for every $M \in \mathsf{Mod}(R)$: 
\begin{itemize}
\item[(a)] $M$ is Gorenstein $\mathcal{B}$-flat.

\item[(b)] There is a short exact sequence of modules 
\[
0 \to  F \to L \to M \to 0
\] 
with $F \in \mathcal{F}(R)$ and $L \in \mathcal{PGF}_{\mathcal{B}}(R)$, which is also $\Hom_R(-,\mathcal{C}(R))$-acyclic, that is, the induced sequence
\[
0 \to \Hom_R(M,C) \to \Hom_R(L,C) \to \Hom_R(F,C) \to 0
\]
of abelian groups is exact for every $C \in \mathcal{C}(R)$.

\item[(c)] $\Ext^1_R(M,C) = 0$ for every $C \in \mathcal{C}(R) \cap (\mathcal{PGF}_{\mathcal{B}}(R))^\perp$.

\item[(d)] There is a short exact sequence of modules 
\[
0 \to M \to F \to L \to 0
\] 
with $F \in \mathcal{F}(R)$ and $L \in \mathcal{PGF}_{\mathcal{B}}(R)$. In particular, we have the equality 
\[
\mathcal{GF}_{\mathcal{B}}(R) \cap (\mathcal{PGF}_{\mathcal{B}}(R))^\perp = \mathcal{F}(R).\footnote{The same result holds in the category $\Ch(R)$ of chain complexes by replacing $\mathcal{F}(R)$, $\mathcal{C}(R)$, $\mathcal{PGF}_{\mathcal{B}}(R)$ and $\mathcal{GF}_{\mathcal{B}}(R)$ by their corresponding analogs in $\Ch(R)$.}
\] 
\end{itemize}
\end{theorem}

\begin{proof} \
\begin{itemize}
\item \underline{Module setting}: Follows by the same arguments employed in \cite[Theorem 3.10]{SarochStovicek}. We only mention some particular points concerning this relative case: 
\begin{itemize}
\item (a) $\Rightarrow$ (b) follows as in \cite{SarochStovicek}.

\item (b) $\Rightarrow$ (c): We shall use an argument slightly different from that in \cite{SarochStovicek}. Suppose we are given a short exact sequence
\[
0 \to  F \to L \to M \to 0
\] 
as in (b). Consider $C \in \mathcal{C}(R) \cap (\mathcal{PGF}_{\mathcal{B}}(R))^\perp$. We have an exact sequence
\[
\Hom_R(L,C) \xrightarrow{\varphi} \Hom_R(F,C) \to \Ext^1_R(M,C) \to \Ext^1_R(L,C)
\]
where $\Ext^1_R(L,C) = 0$ since $L \in \mathcal{PGF}_{\mathcal{B}}(R)$, and $\varphi$ is epic. Hence, $\Ext^1_R(M,C) = 0$. 

\item (c) $\Rightarrow$ (d): Using a pushout argument, along with the inclusion $\mathcal{F}(R) \subseteq (\mathcal{PGF}_{\mathcal{B}}(R))^\perp$ from Proposition \ref{prop:PGFB_definable_R} (2), one can show that every module in ${}^\perp(\mathcal{C}(R) \cap (\mathcal{PGF}_{\mathcal{B}}(R))^\perp)$ has a pure special $\mathcal{PGF}_{\mathcal{B}}(R)$-precover. Now consider a short exact sequence 
\[
0 \to M \to U \to T \to 0
\] 
with $U \in (\mathcal{PGF}_{\mathcal{B}}(R))^\perp$ and $T \in \mathcal{PGF}_{\mathcal{B}}(R)$, resulting from Theorem~\ref{theo:cotorsion_pair_PGFB}. Let $C \in (\mathcal{PGF}_{\mathcal{B}}(R))^\perp$ be a cotorsion module. Then, we have an exact sequence
\[
\Ext^1_R(T,C) \to \Ext^1_R(U,C) \to \Ext^1_R(M,C)
\]
where $\Ext^1_R(T,C) = 0$ and $\Ext^1_R(M,C) = 0$. Then, $U \in {}^\perp(\mathcal{C}(R) \cap (\mathcal{PGF}_{\mathcal{B}}(R))^\perp)$ and so we can consider a pure special $\mathcal{PGF}_{\mathcal{B}}(R)$-precover of $U$, say 
\[
0 \to K \to L \to U \to 0
\] 
with $K \in (\mathcal{PGF}_{\mathcal{B}}(R))^\perp$ and $L \in \mathcal{PGF}_{\mathcal{B}}(R)$. Now take the pullback of $M \to U \leftarrow L$ to get the following commutative diagram with exact rows and columns:
\[
\begin{tikzpicture}[description/.style={fill=white,inner sep=2pt}] 
\matrix (m) [matrix of math nodes, row sep=2.5em, column sep=2.5em, text height=1.25ex, text depth=0.25ex] 
{ 
{} & K & K & {} & {} \\
0 & N & L & T & 0 \\
0 & M & U & T & 0 \\
}; 
\path[>->]
(m-1-2) edge (m-2-2) (m-1-3) edge (m-2-3)
;
\path[->>]
(m-2-2) edge (m-3-2) (m-2-3) edge (m-3-3)
;
\path[->] 
(m-2-1) edge (m-2-2) (m-2-2) edge (m-2-3) (m-2-3) edge (m-2-4) (m-2-4) edge (m-2-5)
(m-3-1) edge (m-3-2) (m-3-2) edge (m-3-3) (m-3-3) edge (m-3-4) (m-3-4) edge (m-3-5) 
; 
\path[-,font=\scriptsize]
(m-1-2) edge [double, thick, double distance=2pt] (m-1-3)
(m-2-4) edge [double, thick, double distance=2pt] (m-3-4)
;
\end{tikzpicture} 
\]
Then, $L \in \mathcal{PGF}_{\mathcal{B}}(R) \cap (\mathcal{PGF}_{\mathcal{B}}(R))^\perp$. It follows that $L$ is a direct summand of a projective module, and so $U$ is a pure epimorphic image of a projective module, which is turn implies that $U \in \mathcal{F}(R)$. Therefore, the implication follows. 

\item (d) $\Rightarrow$ (a): Follows as in \cite{SarochStovicek} but using Lemma~\ref{characterisations} instead of \cite[Lemma 2.4]{BennisGF}. 
\end{itemize}

\item \underline{Chain complex setting}: For the proof in the category $\Ch(R)$ of chain complexes, the implications (b) $\Rightarrow$ (c), (c) $\Rightarrow$ (d) and (d) $\Rightarrow$ (a) follow as above, due to conditions (3) and (4) in Proposition \ref{prop:PGFB_definable_R} and to the chain complex versions of Lemma \ref{characterisations} and Theorem \ref{theo:cotorsion_pair_PGFB}. On the other hand, the implication (a) $\Rightarrow$ (b) in the module case is based, in part, on the following two facts: 
\begin{itemize} 
\item Bravo et al. \cite[Theorem 4.5]{Bravoetal}: $(\Ch(\mathcal{P}(R)),\Ch(\mathcal{P}(R))^\perp)$ is a complete cotorsion pair in $\Ch(R)$, where $\Ch(\mathcal{P}(R))$ denotes the class of complexes of projective modules. 

\item Neeman's \cite[Lemma 8.4]{Neeman}: Every complex of flat modules is a direct limit from complexes in $\Ch(\mathcal{P}(R))$. 
\end{itemize}
One can note after a careful revision of these two results that the arguments employed in \cite{Bravoetal,Neeman} carry over to the category $\Ch(\Ch(R))$ of complexes of chain complexes, that is, if $\mathscr{P}(R)$ denotes the class of projective complexes and $\Ch(\mathscr{P}(R))$ the class of complexes of projective complexes, then $(\Ch(\mathscr{P}(R)),\Ch(\mathscr{P}(R))^\perp)$ is a complete cotorsion pair in $\Ch(\Ch(R))$ and every complex of flat complexes is a direct limit of complexes in $\Ch(\mathscr{P}(R))$. 

After this observation, it is not hard to check that the proof for {\v{S}}aroch and {{\v{S}{t'o}}}v{\'{\i}}{\v{c}}ek's \cite[implication (1) $\Rightarrow$ (2) in Theorem 3.10]{SarochStovicek} also works for the relative case in the context of chain complexes. 
\end{itemize}
\end{proof}

In what follows, let us denote by $\mathscr{C}(R) := (\mathscr{F}(R))^\perp$ the class of all cotorsion complexes. By the previous result, we can write 
\begin{align*}
\mathcal{GF}_\mathcal{B}(R) & = {}^\perp(\mathcal{C}(R) \cap (\mathcal{PGF}_{\mathcal{B}}(R))^\perp) & \text{and} & &  \mathscr{GF}_\mathscr{B}(R) & = {}^\perp(\mathscr{C}(R) \cap (\mathscr{PGF}_{\mathscr{B}}(R))^\perp),
\end{align*}
and hence $\mathcal{GF}_{\mathcal{B}}(R)$ and $\mathscr{GF}_\mathscr{B}(R)$ are both closed under extensions, whenever $\mathcal{B}$ and $\mathscr{B}$ are semi-definable classes. Along with some other properties of $\mathcal{GF}_{\mathcal{B}}(R)$ and $\mathscr{GF}_\mathscr{B}(R)$, the latter will allow us to construct hereditary complete cotorsion pairs 
\begin{align*}
(\mathcal{GF}_{\mathcal{B}}(R),(\mathcal{GF}_{\mathcal{B}}(R))^\perp) & & \text{and} & & (\mathscr{GF}_{\mathscr{B}}(R),(\mathscr{GF}_{\mathscr{B}}(R))^\perp)
\end{align*} 
in $\mathsf{Mod}(R)$ and $\mathsf{Ch}(R)$, respectively. We shall also construct the hereditary complete cotorsion pair 
\[
(\mathscr{GF}^{\otimes^.}_{\mathscr{B}}(R),(\mathscr{GF}^{\otimes^.}_{\mathscr{B}}(R))^\perp)
\] 
in the case where $\mathscr{GF}^{\otimes^.}_{\mathscr{B}}(R)$ is closed under extensions. Before constructing these pairs, let us show that $\mathcal{GF}_{\mathcal{B}}(R)$ and $\mathscr{GF}_{\mathscr{B}}(R)$ are precovering Kaplansky classes of modules and complexes for any choice of $\mathcal{B} \subseteq \mathsf{Mod}(R)$ and $\mathscr{B} \subseteq \mathsf{Ch}(R)$. For this property, there will be a remarkable difference between the proofs for $\mathsf{Mod}(R)$ and $\mathsf{Ch}(R)$, so we shall state and prove these results separately.

\begin{proposition}\label{Kaplansky}
The class $\mathcal{GF}_{\mathcal{B}}(R)$ of Gorenstein $\mathcal{B}$-flat modules is a precovering Kaplansky class for any class $\mathcal{B}$ of right $R$-modules.
\end{proposition}

\begin{proof}
Let ${}_{\mathcal{B}}\widetilde{\mathcal{F}}$ be the class of exact and $(\mathcal{B} \otimes_R -)$-acyclic complexes of flat modules. By \cite[Theorem 3.7]{EG15}, the class ${}_{\mathcal{B}}\widetilde{\mathcal{F}}$ is a Kaplansky class.  Moreover, ${}_{\mathcal{B}}\widetilde{\mathcal{F}}$ is also  closed under direct limits, extensions, direct summands, and contains a generator for $\Ch(R)$. By Gillespie's \cite[Proposition 4.8]{GillespieKaplansky}, we have that $({}_{\mathcal{B}}\widetilde{\mathcal{F}},({}_{\mathcal{B}}\widetilde{\mathcal{F}})^{\perp})$ is a small cotorsion pair in $\Ch(R)$, and so cogenerated by a set, which in turn implies completeness since $\Ch(R)$ is a Grothendieck category with enough projective objects. Finally, since the hypotheses in Aldrich et al. \cite[Corollaries 2.11, 2.12 and 2.13]{AEGO} are satisfied, we have that $({}_{\mathcal{B}}\widetilde{\mathcal{F}},({}_{\mathcal{B}}\widetilde{\mathcal{F}})^{\perp})$ is a perfect cotorsion pair. Therefore, we can note that the class $\mathcal{GF}_{\mathcal{B}}(R)$ is Kaplansly (as it is the class of $0$-cycles of complexes in the Kaplansky class ${}_{\mathcal{B}}\widetilde{\mathcal{F}}$) and precovering (proceeding as in Yang and Liang \cite[Theorem A]{YangLiang}).
\end{proof}

We shall need the chain complex version of \cite[Theorem 3.7]{EG15} in order to show that $\mathscr{GF}_{\mathscr{B}}(R)$ is a precovering Kaplansky class. Below we shall show that the class ${}_{\mathscr{B}}\widetilde{\mathscr{F}}$ of exact and $(\mathscr{B} \otimes -)$-acyclic complexes of flat complexes is a Kaplansky class for any $\mathscr{B} \subseteq \Ch(R\op)$. 

The next two statements correspond to more general properties for closed symmetric monoidal Grothendieck categories proved in \cite{BEGIP}. We present sketches of their proofs for the setting of complexes of left and right $R$-modules.

\begin{lemma}\label{lem:Banff1}
Let $X_\bullet$ be a complex in ${}_{\mathscr{B}} \widetilde{\mathscr{F}}$ and $X'_\bullet$ be an exact subcomplex of $X_\bullet$ such that $Z_m(X'_\bullet) \subseteq Z_m(X_\bullet)$ and $X'_m \subseteq X_m$ are pure containments for all $m \in \mathbb{Z}$. Then, the complexes of complexes $X'_\bullet$ and $X_\bullet / X'_\bullet$ belong to ${}_{\mathscr{B}}\widetilde{\mathscr{F}}$. 
\end{lemma}

\begin{proof}
We split the proof into three parts:
\begin{enumerate}
\item The complexes $X'_m, X_m / X'_m \in \Ch(R)$ are flat: This follows by applying \cite[Lemma 4.7]{GillespieFlat} to the pure containment $X'_m \subseteq X_m$.  

\item The complex $X'_\bullet$ belongs to ${}_{\mathscr{B}}\widetilde{\mathscr{F}}$: Consider the following commutative diagram where the vertical arrows are pure embeddings: 
\begin{equation}\label{fig4} 
\parbox{4in}{
\begin{tikzpicture}[description/.style={fill=white,inner sep=2pt}] 
\matrix (m) [ampersand replacement=\&, matrix of math nodes, row sep=3em, column sep=2.5em, text height=1.25ex, text depth=0.25ex] 
{ 
\eta'_m \colon 0 \& Z_m(X'_\bullet) \& X'_m \& Z_{m-1}(X'_\bullet) \& 0 \\
\eta_m \colon 0 \& Z_m(X_\bullet) \& X_m \& Z_{m-1}(X_\bullet) \& 0 \\
}; 
\path[->] 
(m-1-1) edge (m-1-2) (m-1-2) edge (m-1-3) (m-1-3) edge (m-1-4) (m-1-4) edge (m-1-5)
(m-2-1) edge (m-2-2) (m-2-2) edge (m-2-3) (m-2-3) edge (m-2-4) (m-2-4) edge (m-2-5)
; 
\path[right hook->]
(m-1-2) edge node[above,sloped] {\scriptsize$\text{pure}$} (m-2-2) (m-1-3) edge node[above,sloped] {\scriptsize$\text{pure}$} (m-2-3) (m-1-4) edge node[above,sloped] {\scriptsize$\text{pure}$} (m-2-4)
;
\end{tikzpicture} 
}
\end{equation}
Now let $B \in \mathscr{B}$ and apply the functor $B \otimes -$ to \eqref{fig4}. We obtain the commutative diagram
\[
\begin{tikzpicture}[description/.style={fill=white,inner sep=2pt}] 
\matrix (m) [matrix of math nodes, row sep=3em, column sep=2.5em, text height=1.25ex, text depth=0.25ex] 
{ 
B \otimes \eta'_m \colon 0 & B \otimes Z_m(X'_\bullet) & B \otimes X'_m & B \otimes Z_{m-1}(X'_\bullet) & 0 \\
B \otimes \eta_m \colon 0 & B \otimes Z_m(X_\bullet) & B \otimes X_m & B \otimes Z_{m-1}(X_\bullet) & 0 \\
}; 
\path[->] 
(m-1-1) edge (m-1-2) (m-1-2) edge (m-1-3) (m-1-3) edge (m-1-4) (m-1-4) edge (m-1-5)
(m-2-1) edge (m-2-2) (m-2-2) edge (m-2-3) (m-2-3) edge (m-2-4) (m-2-4) edge (m-2-5)
; 
\path[>->]
(m-1-2) edge (m-2-2) (m-1-3) edge (m-2-3) (m-1-4) edge (m-2-4)
;
\end{tikzpicture} 
\]
where $B \otimes \eta_m$ is exact since $X_\bullet \in {}_{\mathscr{B}}\widetilde{\mathscr{F}}$. Moreover, the sequence 
\[
B \otimes Z_m(X'_\bullet) \to B \otimes X'_m \to B \otimes Z_{m-1}(X'_\bullet) \to 0
\] 
is exact since the functor $B \otimes -$ is right exact, and $0 \to B \otimes Z_m(X'_\bullet) \to B \otimes X'_m$ is monic by diagram chasing. It follows that the sequence $B \otimes \eta'_m$ is exact for every $m \in \mathbb{Z}$, that is, $X'_\bullet \in {}_{\mathscr{B}} \widetilde{\mathscr{F}}$. 

\item The complex $X_\bullet / X'_\bullet$ belongs to ${}_{\mathscr{B}}\widetilde{\mathscr{F}}$: First, note that $X_\bullet / X'_\bullet$ is exact as it is a quotient of exact complexes. Then, there is a short exact sequence for every $m \in \mathbb{Z}$:
\[
0 \to Z_m(X_\bullet / X'_\bullet) \to X_m / X'_m \to Z_{m-1}(X_\bullet / X'_\bullet) \to 0.
\]
Consider the following commutative diagram for every $B \in \mathscr{B}$:
\begin{equation}\label{fig6} 
\parbox{5in}{
\begin{tikzpicture}[description/.style={fill=white,inner sep=2pt}] 
\matrix (m) [ampersand replacement=\&, matrix of math nodes, row sep=2.5em, column sep=2em, text height=1.25ex, text depth=0.25ex] 
{ 
0 \& B \otimes Z_m(X'_\bullet) \& B \otimes X'_m \& B \otimes Z_{m-1}(X'_\bullet) \& 0 \\
0 \& B \otimes Z_m(X_\bullet) \& B \otimes X_m \& B \otimes Z_{m-1}(X_\bullet) \& 0 \\
0 \& B \otimes Z_m(X_\bullet / X'_\bullet) \& B \otimes (X_m / X'_m) \& B \otimes Z_{m-1}(X_\bullet / X'_\bullet) \& 0 \\
}; 
\path[->] 
(m-1-1) edge (m-1-2) (m-1-2) edge (m-1-3) (m-1-3) edge (m-1-4) (m-1-4) edge (m-1-5)
(m-2-1) edge (m-2-2) (m-2-2) edge (m-2-3) (m-2-3) edge (m-2-4) (m-2-4) edge (m-2-5)
(m-3-1) edge (m-3-2) (m-3-2) edge (m-3-3) (m-3-3) edge (m-3-4) (m-3-4) edge (m-3-5)
; 
\path[>->]
(m-1-2) edge (m-2-2) (m-1-3) edge (m-2-3) (m-1-4) edge (m-2-4)
;
\path[->>]
(m-2-2) edge (m-3-2) (m-2-3) edge (m-3-3) (m-2-4) edge (m-3-4)
;
\end{tikzpicture} 
}
\end{equation}
The first and second rows of \eqref{fig6} are exact since $X_\bullet, X'_\bullet \in {}_{\mathscr{B}}\widetilde{\mathscr{F}}$, while the columns are exact since $X'_m \subseteq X_m$ and $Z_m(X'_\bullet) \subseteq Z_m(X_\bullet)$ are pure containments. It follows that the bottom row is exact, and hence $X_\bullet / X'_\bullet \in {}_{\mathscr{B}} \widetilde{\mathscr{F}}$. 
\end{enumerate}
\end{proof}

\begin{lemma}\label{lem:Banff2}
The class ${}_{\mathscr{B}} \widetilde{\mathscr{F}}$ is a Kaplansky class for every class $\mathscr{B} \subseteq \Ch(R\op)$ of complexes of right $R$-modules.
\end{lemma}

\begin{proof}
We first note that by \cite[Theorem 2.33]{AdamekRosicky}, there exists a regular cardinal $\kappa$ for the category $\Ch(\Ch(R))$ such that every $\kappa$-presentable subcomplex $X_\bullet \subseteq Y_\bullet$ of a complex $Y_\bullet \in {}_{\mathscr{B}}\widetilde{\mathscr{F}}$ is contained in a $\kappa$-pure subcomplex $X'_\bullet \subseteq Y_\bullet$ with $X'_\bullet$ also $\kappa$-presentable. In what follows, we shall show that $X'_\bullet$ and $Y_\bullet / X'_\bullet$ belong to ${}_{\mathscr{B}}\widetilde{\mathscr{F}}$, that is, that ${}_{\mathscr{B}}\widetilde{\mathscr{F}}$ is $\kappa$-Kaplansky. The latter will be a consequence of Lemma~\ref{lem:Banff1}. We divide the proof into several steps:
\begin{enumerate}
\item $X'_m$ is a pure subcomplex of $Y_m$: We use the fact that every $\kappa$-pure subcomplex is a pure subcomplex. Thus, we only need to show that 
\[
0 \to X'_m \to Y_m
\] 
is a $\kappa$-pure monomorphism, that is, for any commutative diagram 
\begin{equation}\label{fig7} 
\parbox{2in}{
\begin{tikzpicture}[description/.style={fill=white,inner sep=2pt}] 
\matrix (m) [ampersand replacement=\&, matrix of math nodes, row sep=2.5em, column sep=3em, text height=1.25ex, text depth=0.25ex] 
{ 
{} \& W \& Z \\ 
0 \& X'_m \& Y_m \\
}; 
\path[->] 
(m-1-2) edge (m-1-3) edge (m-2-2)
(m-2-1) edge (m-2-2) (m-2-2) edge (m-2-3)
(m-1-3) edge (m-2-3)
; 
\path[dotted,->]
(m-1-3) edge (m-2-2)
;
\end{tikzpicture} 
}
\end{equation}
with $W$ and $Z$ $\kappa$-presentable, there exists a lifting $Z \to X'_m$ making the upper left triangle commutative. This diagram induces the following commutative diagram in $\Ch(\Ch(R))$:
\begin{equation}\label{fig8} 
\parbox{2in}{
\begin{tikzpicture}[description/.style={fill=white,inner sep=2pt}] 
\matrix (m) [ampersand replacement=\&, matrix of math nodes, row sep=2.5em, column sep=2em, text height=1.25ex, text depth=0.25ex] 
{ 
{} \& D^m(W) \& D^m(Z) \\ 
0 \& X'_\bullet \& Y_\bullet \\
}; 
\path[->] 
(m-1-2) edge (m-1-3) edge (m-2-2)
(m-2-1) edge (m-2-2) (m-2-2) edge (m-2-3)
(m-1-3) edge (m-2-3)
(m-1-3) edge (m-2-2)
; 
\end{tikzpicture} 
}
\end{equation}
where $D^m(W)$ and $D^m(Z)$ are $\kappa$-presentable since the functors
\begin{align*}
\Hom(D^m(W),-) & \cong \Hom(W,(-)_m), \\
\Hom(D^m(Z),-) & \cong \Hom(Z,(-)_m),
\end{align*}
preserve $\kappa$-directed colimits, and $\kappa$-directed colimits in $\Ch(\Ch(R))$ are computed componentwise. It follows that there exists a morphism $f_\bullet \colon D^m(Z) \to X'_\bullet$ making the upper left triangle of \eqref{fig8} commutative, since the containment $X'_\bullet \subseteq Y_\bullet$ is $\kappa$-pure. Hence, the $m$-th component $f_m \colon Z \to X'_m$ yields the desired lifting in \eqref{fig7}. 

\item $Z_m(X'_\bullet)$ is a pure subcomplex of $Z_m(Y_\bullet)$: As above, one needs to show that the inclusion $0 \to Z_m(X'_\bullet) \to Z_m(Y_\bullet)$ is a $\kappa$-pure monomorphism in $\Ch(R)$. This follows by noticing that 
\[
\Hom(S^m(W),-) \cong \Hom(W,Z_m(-))
\] 
preserves $\kappa$-directed colimits. 

\item $X'_\bullet$ is exact: First, since the containment $X'_\bullet \subseteq Y_\bullet$ is $\kappa$-pure, note by (2) above that the sequence 
\[
0 \to Z_m(X'_\bullet) \to Z_m(Y_\bullet) \to Z_m(Y_\bullet / X'_\bullet) \to 0
\] 
is $\kappa$-pure for every $m \in \mathbb{Z}$. In particular, since the category $\Ch(R)$ has a generating family of $\kappa$-presentable objects, it follows that the previous sequence is exact. Thus, we can consider the following commutative diagram with pure exact rows:
\[
\begin{tikzpicture}[description/.style={fill=white,inner sep=2pt}] 
\matrix (m) [ampersand replacement=\&, matrix of math nodes, row sep=3em, column sep=2.5em, text height=1.25ex, text depth=0.25ex] 
{ 
0 \& Z_m(X'_\bullet) \& Z_m(Y_\bullet) \& Z_m(Y_\bullet / X'_\bullet) \& 0 \\
0 \& X'_m \& Y_m \& Y_m / X'_m \& 0 \\
}; 
\path[->] 
(m-1-1) edge (m-1-2) (m-1-2) edge (m-1-3) (m-1-3) edge (m-1-4) (m-1-4) edge (m-1-5)
(m-2-1) edge (m-2-2) (m-2-2) edge (m-2-3) (m-2-3) edge (m-2-4) (m-2-4) edge (m-2-5)
; 
\path[right hook->]
(m-1-2) edge (m-2-2) (m-1-3) edge (m-2-3) (m-1-4) edge (m-2-4)
;
\end{tikzpicture} 
\]
The Snake Lemma applied to the previous diagram yields a short exact sequence 
\[
0 \to B_{m-1}(X'_\bullet) \to B_{m-1}(Y_\bullet) \to B_{m-1}(Y_\bullet / X'_\bullet) \to 0
\] 
which can be embedded into the pure exact sequence 
\[
0 \to Z_{m-1}(X'_\bullet) \to Z_{m-1}(Y_\bullet) \to Z_{m-1}(Y_\bullet / X'_\bullet) \to 0
\]
as in the following commutative diagram:
\[
\begin{tikzpicture}[description/.style={fill=white,inner sep=2pt}] 
\matrix (m) [matrix of math nodes, row sep=3em, column sep=2.5em, text height=1.25ex, text depth=0.25ex] 
{ 
0 & B_{m-1}(X'_\bullet) & B_{m-1}(Y_\bullet) & B_{m-1}(Y_\bullet / X'_\bullet) & 0 \\
0 & Z_{m-1}(X'_\bullet) & Z_{m-1}(Y_\bullet) & Z_{m-1}(Y_\bullet / X'_\bullet) & 0 \\
}; 
\path[->] 
(m-1-1) edge (m-1-2) (m-1-2) edge (m-1-3) (m-1-3) edge (m-1-4) (m-1-4) edge (m-1-5)
(m-2-1) edge (m-2-2) (m-2-2) edge (m-2-3) (m-2-3) edge (m-2-4) (m-2-4) edge (m-2-5)
; 
\path[right hook->]
(m-1-2) edge (m-2-2) (m-1-4) edge (m-2-4)
;
\path[-,font=\scriptsize]
(m-1-3) edge [double, thick, double distance=2pt] node[above,sloped] {\footnotesize$\sim$} (m-2-3)
;
\end{tikzpicture} 
\]
Applying the Snake Lemma again shows that the left-hand side arrow is an isomorphism. 
\end{enumerate}
\end{proof}

\begin{proposition}\label{prop:GFB_properties_Ch}
The classes $\mathscr{GF}_{\mathscr{B}}(R)$ and $\mathscr{GF}^{\otimes^.}_{\mathscr{B}}(R)$ of Gorenstein $\mathscr{B}$-flat complexes and Gorenstein $\mathscr{B}$-flat complexes under $\otimes^.$ are precovering Kaplansky classes for every class $\mathscr{B} \subseteq \Ch(R\op)$. 
\end{proposition}

\begin{proof}
The proof is similar to the one appearing in Proposition \ref{Kaplansky}. We just point out some details concerning some arguments in the category $\Ch(\Ch(R))$ of chain complexes of complexes. First, it is not hard to check that the $\otimes^.$-versions of Lemmas \ref{lem:Banff1} and \ref{lem:Banff2} are also valid. Moreover, for the $\otimes^.$-version of Lemma \ref{lem:Banff1}, it is not required that $Z_m(X'_\bullet) \subseteq Z_m(X_\bullet)$ is a pure containment for every $m \in \mathbb{Z}$. Applying  Lemma \ref{lem:Banff2} and its $\otimes^.$-version, we get that ${}_{\mathscr{B}} \widetilde{\mathscr{F}}$ and ${}_{\mathscr{B}} \widetilde{\mathscr{F}}^{\otimes^.}$ (the class of exact and $(\mathscr{B} \otimes^. -)$-acyclic complexes of $\otimes^.$-flat complexes) are Kaplansky classes. As in the proof of  Proposition \ref{Kaplansky}, we can argue that ${}_{\mathscr{B}} \widetilde{\mathscr{F}}$ and ${}_{\mathscr{B}} \widetilde{\mathscr{F}}^{\otimes^.}$ are the left halves of two perfect cotorsion pairs in $\Ch(\Ch(R))$. Again, as in \cite[Theorem A]{YangLiang}, we can conclude that $\mathscr{GF}_{\mathscr{B}}(R)$ and $\mathscr{GF}^{\otimes^.}_{\mathscr{B}}(R)$ are precovering Kaplansky classes.
\end{proof}

We continue this section constructing several approximations by the classes $\mathcal{GF}_{\mathcal{B}}(R)$, $\mathscr{GF}_{\mathscr{B}}(R)$ and $\mathscr{GF}^{\otimes^.}_{\mathscr{B}}(R)$ in the case they are closed under extensions. In particular, we shall have approximations by the first two classes if $\mathcal{B}$ and $\mathscr{B}$ are semi-definable.

\begin{proposition}\label{prop:GFB_properties}
The following assertions hold in case $\mathcal{GF}_{\mathcal{B}}(R)$ is closed under extensions:
\begin{enumerate}
\item $\mathcal{GF}_{\mathcal{B}}(R)$ is resolving and closed under direct limits. As a consequence, $\mathcal{GF}_{\mathcal{B}}(R)$ is a covering class.

\item If in addition, $\mathcal{GF}_{\mathcal{B}}(R)$ is closed under direct products, then $\mathcal{GF}_{\mathcal{B}}(R)$ is preenveloping.\footnote{Properties (1) and (2) also hold for the classes $\mathscr{GF}_{\mathscr{B}}(R)$ and $\mathscr{GF}_{\mathscr{B}}^{\otimes^.}(R)$, provided they are closed under extensions.} 
\end{enumerate}
\end{proposition}

\begin{proof} \
\begin{enumerate}
\item For the module case, the proof that $\mathcal{GF}_{\mathcal B}(R)$ is closed under taking kernels of epimorphisms follows from a similar argument to that in \cite[Lemma 4.7]{BEI17} with the class of Gorenstein $\mathcal{B}$-flat modules replacing that of Gorenstein AC-flat modules. The proof of closure under direct limits uses Lemma \ref{characterisations}  and the same argument as in \cite[Lemma 3.1]{gang:12:gorflat} with the class $\mathcal{B}$ replacing the class of injectives.

For the proof concerning the classes $\mathscr{GF}_{\mathscr{B}}(R)$ and $\mathscr{GF}_{\mathscr{B}}^{\otimes^.}(R)$ we can use the same argument as in \cite[Lemma 3.1]{gang:12:gorflat}, applying Lemma \ref{characterisations}, with $\mathscr{B}$ replacing the injective complexes.

\item By Proposition \ref{Kaplansky}, the class $\mathcal{GF}_{\mathcal{B}}(R)$ is Kaplansky. Since it is closed under extensions, we get from part (1) that it is closed under direct limits. Therefore the result follows from Enochs and L\'opez Ramos \cite[Theorem 2.5]{EEnJLR02}. 

In the category of complexes, $\mathscr{GF}_{\mathscr{B}}(R)$ and $\mathscr{GF}_{\mathscr{B}}^{\otimes^.}(R)$ are Kaplansky classes assumed to be closed under extensions, and so they are also closed under direct limits by part (1). Then by \v St'ov\'\i \v cek's \cite[Corollary 2.7]{Stovicek}, $\mathscr{GF}_{\mathscr{B}}(R)$ and $\mathscr{GF}_{\mathscr{B}}^{\otimes^.}(R)$ are deconstructible classes. Hence, again by \cite[Theorem on pg. 2]{Stovicek}, we can conclude that $\mathscr{GF}_{\mathscr{B}}(R)$ and $\mathscr{GF}_{\mathscr{B}}^{\otimes^.}(R)$ are preenveloping.
\end{enumerate}
\end{proof}

We shall denote by $\mathcal{GC}_{\mathcal B}(R)$ the right orthogonal class $(\mathcal{GF}_{\mathcal B}(R))^\perp$. We call a module in $\mathcal{GC}_{\mathcal B}(R)$ a \emph{Gorenstein $\mathcal B$-cotorsion module}. Similarly, we set the notations $\mathscr{GC}_{\mathscr{B}}(R) := (\mathscr{GF}_{\mathscr{B}}(R))^\perp$ and $\mathscr{GC}^{\otimes^.}_{\mathscr{B}}(R) := (\mathscr{GF}_{\mathscr{B}}^{\otimes^.}(R))^\perp$ for \emph{Gorenstein $\mathscr{B}$-cotorsion complexes} and \emph{Gorenstein $\mathscr{B}$-cotorsion complexes under $\otimes^.$}, respectively.

\begin{corollary}\label{complete}
If $\mathcal{GF}_{\mathcal{B}}(R)$ is closed under extensions, then the pair 
\[
(\mathcal{GF}_{\mathcal B}(R),\mathcal{GC}_{\mathcal{B}}(R))
\] 
is a hereditary complete cotorsion pair in $\Mod(R)$ (cogenerated by a set).\footnote{Similarly, there are two hereditary complete cotorsion pairs $(\mathscr{GF}_{\mathscr{B}}(R), \mathscr{GC}_{\mathscr{B}}(R))$ and $(\mathscr{GF}^{\otimes^.}_{\mathscr{B}}(R),\mathscr{GC}_{\mathscr{B}}^{\otimes^.}(R))$, in the cases where $\mathscr{GF}_{\mathscr{B}}(R)$ and $\mathscr{GF}_{\mathscr{B}}^{\otimes^.}(R)$ are closed under extensions.} 
\end{corollary}

\begin{proof}
By Proposition \ref{Kaplansky} the class $\mathcal{GF}_{\mathcal B}(R)$ is a Kaplanksy class. By part (1) of Proposition \ref{prop:GFB_properties}, it is also closed under direct limits. Therefore, since it is closed under extensions, a standard argument on deconstruction and transfinite induction yields that the cotorsion pair $(\mathcal{GF}_{\mathcal B}(R),\mathcal{GC}_{\mathcal{B}}(R))$ is cogenerated by a set.

In the setting of chain complexes let us use a different argument. We have by the chain complex version of Proposition \ref{prop:GFB_properties} that $\mathscr{GF}_{\mathscr{B}}(R)$ and $\mathscr{GF}_{\mathscr{B}}^{\otimes^.}(R)$ are covering classes which are also resolving. Since $\Ch(R)$ has enough projectives, we can note that covers by $\mathscr{GF}_{\mathscr{B}}(R)$ and $\mathscr{GF}_{\mathscr{B}}^{\otimes^.}(R)$ are epic. Using again the assumption that $\mathscr{GF}_{\mathscr{B}}(R)$ is closed under extensions, along with Wakamatsu Lemma, we have that $\mathscr{GF}_{\mathscr{B}}(R)$ and $\mathscr{GF}_{\mathscr{B}}^{\otimes^.}(R)$ are special precovering. For a proof of this lemma that works in any abelian category, see \cite[Lemma 2.1.13]{GT}. The rest of the proof is straightforward. 
\end{proof}

Let us present some examples of relative Gorenstein flat cotorsion pairs derived from the previous corollary.

\begin{example}\label{ex:Gorenstein_AC} \
\begin{enumerate}
\item We obtain the already known result \cite[Corollary 3.11]{SarochStovicek} that the class of Gorenstein flat modules is the left half of a hereditary complete cotorsion pair. 

\item Recall that a module $M \in \Mod(R)$ is \emph{level} if $\Tor^R_1(F,M) = 0$ for every $F \in \mathcal{FP}_\infty(R^{\rm op})$. Let us denote by $\mathcal{L}(R)$ the class of level modules. 

Consider the class $\mathcal{GF}_{\rm AC}(R)$ of Gorenstein AC-flat modules from Example \ref{ex:GFB_cases}. The class $\mathcal{AC}(R\op)$ of absolutely clean right $R$-modules turns out to be semi-definable. Indeed, we have by \cite[Proposition 2.7]{BGH} that $\langle \mathcal{AC}(R\op) \rangle = \mathcal{AC}(R\op)$, and the duality pair $({\rm Prod}(\langle \mathcal{AC}(R\op) \rangle^+)^{\rm p},\mathcal{AC}(R\op))$ from Theorem~\ref{theo:definable_duality_pair} coincides with the duality pair $(\mathcal{L}(R),\mathcal{AC}(R\op))$ constructed in \cite{BGH}. For, let us check that ${\rm Prod}(\langle \mathcal{AC}(R\op) \rangle^+)^{\rm p} = \mathcal{L}(R)$. First, note that any $A \in {\rm Prod}(\langle \mathcal{AC}(R\op) \rangle^+)^{\rm p}$ is a pure submodule of a module $A' = \prod_{i \in I} E^+_i$ with $E_i \in \mathcal{AC}(R^{\rm op})$. Since $E^+_i \in \mathcal{L}(R)$ for every $i \in I$, and $\mathcal{L}(R)$ is definable, we have that $A \in \mathcal{L}(R)$. On the other hand, for every level module $L \in \mathcal{L}(R)$ we have that $L^{++} \in {\rm Prod}(\langle \mathcal{AC}(R\op) \rangle^+)^{\rm p}$ since $L^+ \in \mathcal{AC}(R^{\rm op})$. Then $L \in {\rm Prod}(\langle \mathcal{AC}(R\op) \rangle^+)^{\rm p}$ by Lemma~\ref{lem:Prest}. 

Hence, we have the following properties for Gorenstein AC-flat modules:
\begin{itemize}
\item We have a hereditary complete cotorsion pair 
\[
(\mathcal{GF}_{\rm AC}(R),(\mathcal{GF}_{\rm AC}(R))^\perp).
\]
\end{itemize}

Recall from \cite[Definition 4.2]{BEI17} that a module $M \in \Mod(R)$ is \emph{strongly Gorenstein AC-flat} if $M$ is a cycle of an exact and $(\mathcal{AC}(R\op) \otimes_R -)$-acyclic complex of the form 
\[
\cdots \to F \to F \to F \to \cdots
\] 
where $F \in \mathcal{F}(R)$.

\begin{itemize}
\item \cite[Theorem 4.9]{BEI17}: For any ring $R$, an $R$-module $M$ is Gorenstein AC-flat if, and only if, it is a direct summand of a strongly Gorenstein AC-flat module.  

\item \cite[Proposition 4.6]{BEI17} or part (1) of Proposition \ref{prop:GFB_properties}: Every module has a Gorenstein AC-flat cover.
\end{itemize}

\item Let us consider an intermediate situation arising when studying the class $\mathcal{I}_n(R\op)$ of \emph{$\text{FP}_n$-injective} right $R$-modules defined in \cite[Definition 3.1]{BP17}. This class is the right orthogonal complement $(\mathcal{FP}_n(R^{\rm op}))^\perp$, where $\mathcal{FP}_n(R^{\rm op})$ denotes the class of right $R$-modules \emph{of type $\text{FP}_n$}, that is, those $N \in \Mod(R^{\rm op})$ for which there is an exact sequence
\[
P_n \to P_{n-1} \to \cdots \to P_1 \to P_0 \to N \to 0
\]
where $P_k$ is finitely generated and projective for every $0 \leq k \leq n$. 

Let us say that a module is \emph{Gorenstein $\text{FP}_n$-flat} if it is Gorenstein $\mathcal{I}_n(R\op)$-flat. By \cite[Proposition 3.10]{BP17}, we have that $\mathcal{I}_n(R\op)$ is a definable class if $n > 1$. Thus, if $\mathcal{GF}_{\textrm{FP}_n}(R)$ denotes the class of Gorenstein $\text{FP}_n$-flat modules, we have that $\mathcal{GF}_{\textrm{FP}_n}(R)$ is closed under extensions. As a consequence of the previous results, we have that 
\[
(\mathcal{GF}_{\textrm{FP}_n}(R),(\mathcal{GF}_{\textrm{FP}_n}(R))^\perp)
\] 
is a hereditary perfect cotorsion pair. 

The case $n = 1$, on the other hand, is already covered in (1) above, since Gorenstein flat modules coincide with Gorenstein $\mathcal{I}_1(R\op)$-modules by \cite[Lemma 5.3]{EG15}.
\end{enumerate}
\end{example}


\subsection*{Complexes of Gorenstein $\mathcal B$-flat modules}

We conclude this section exploring the relation between relative Gorenstein flat complexes and relative Gorenstein flat modules. Given a class of right $R$-modules $\mathcal{B}$, we can consider the class $\Ch(\mathcal{GF}_{\mathcal B}(R))$ of complexes of Gorenstein $\mathcal{B}$-flat modules. In light of the previous section, it makes sense to wonder whether or not $\Ch(\mathcal{GF}_{\mathcal B}(R))$ agrees either with the class of Gorenstein $\mathscr{D}$-flat complexes or Gorenstein $\mathscr{D}$-flat complexes under $\otimes^.$, for a suitable class of complexes $\mathscr{D} \subseteq \Ch(R\op)$. For the first possibility, we can establish the following general relation between $\Ch(\mathcal{GF}_{\mathcal B}(R))$ and $\mathscr{GF}_{\Ch(\mathcal{B})}(R)$.

\begin{lemma}\label{relation.complexes}
Let $\mathcal{B}$ be a class of right $R$-modules, and consider the class $\Ch(\mathcal{B})$ of complexes of modules from $\mathcal{B}$. Then, every Gorenstein $\Ch(\mathcal{B})$-flat complex is a complex of Gorenstein $\mathcal{B}$-flat modules.
\end{lemma}

\begin{proof}
Let $X$ be a Gorenstein $\Ch(\mathcal{B})$-flat complex. Then, there exists an exact sequence of flat complexes 
\[
F_\bullet \colon \cdots \rightarrow F_1 \rightarrow F_0 \rightarrow F_{-1} \rightarrow \cdots
\]
such that $X = \Ker (F_0 \rightarrow F_{-1})$ and $D \otimes F_\bullet$ is exact for any complex $D \in \Ch(\mathcal{B})$. For any $m \in \mathbb{Z}$ note that we have an exact sequence of flat modules 
\[
(F_\bullet)_m \colon \cdots \rightarrow F_{1,m} \rightarrow F_{0,m} \rightarrow F_{-1,m} \rightarrow \cdots
\] 
such that $X_m = \Ker (F_{0,m} \rightarrow F_{-1,m})$. For any right $R$-module $B \in \mathcal{B}$, we have a natural isomorphism $B \otimes_R (F_\bullet)_m \simeq D^m(B) \otimes F_\bullet$, where $D^m(B) \otimes F_\bullet$ is exact, and then so is $B \otimes_R (F_\bullet)_m$. Thus each $X_m$ is a Gorenstein $\mathcal{B}$-flat module.
\end{proof}

The following proposition is based on a result by Yang and Liu \cite[Corollary 3.12]{gang:12:gorflatGF}.

\begin{proposition}\label{dwGBflat.complexes}
Let $\mathcal B$ be a class of right $R$-modules and assume that $\mathcal{GF}_{{\mathcal B}}(R)$ is closed under extensions. We have the equality
\[
\mathscr{GF}_{\widehat{\mathcal{B}}}(R) = \Ch(\mathcal{GF}_{{\mathcal{B}}}(R))
\] 
where
\[
\widehat{\mathcal{B}} := \left\{ X \in \Ch(R\op) \mbox{ {\rm :} } X \simeq \bigoplus_{m \in \mathbb{Z}} D^m(B_m),\ B_m\in \mathcal{B} \right\}.
\]
In particular, if $\mathcal{B}$ is the class of injective right $R$-modules, we recover  \cite[Corollary 3.12]{gang:12:gorflatGF}.
\end{proposition}

\begin{proof}
The containment $(\subseteq)$ is Lemma \ref{relation.complexes}. For the converse containment $(\supseteq)$, the same arguments as in \cite[Lemmas 3.4, 3.8 and 3.9]{gang:12:gorflatGF} work, replacing the class of Gorenstein flat modules with that of Gorenstein $\mathcal{B}$-flat modules.
\end{proof}

There are classes of relative Gorenstein flat complexes that cannot be written as a class of complexes of relative Gorenstein flat modules, as we show below.

\begin{remark}\label{rem:strict_inclusion}
Consider the case where $\mathcal{B}$ is the class $\mathcal{AC}(R\op)$ of absolutely clean right $R$-modules in the previous proposition. We shall see that the class $\mathscr{GF}_{\rm AC}(R)$ of Gorenstein AC-flat complexes (see Example \ref{ex:GFB_cases} (4)) is not necessarily the class of complexes of Gorenstein AC-flat modules.

Recall from \cite{BravoGillespie} that a complex is absolutely clean if it is exact and each of its cycles is an absolutely clean module. Thus, one can note that the class $\mathscr{AC}(R\op)$ of absolutely clean complexes is not necessarily the class $\widehat{\mathcal{AC}(R\op)}$. In fact, we have that the following are equivalent:
\begin{itemize}
\item[(a)] $\mathscr{AC}(R\op) = \widehat{\mathcal{AC}(R\op)}$.

\item[(b)] $R$ is a right Noetherian ring.

\item[(c)] Every module in $\mathcal{AC}(R\op)$ is a direct summand of a $\mathcal{FP}_\infty(R)$-filtered module. 
\end{itemize}

For the implication (b) $\Rightarrow$ (a), if $R$ is right Noetherian then $\mathcal{AC}(R\op)$ and $\mathscr{AC}(R\op)$ coincide with the classes $\mathcal{I}(R\op)$ and $\mathscr{I}(R\op)$ of injective modules and injective complexes, by \cite[Proposition 2.1 (1)]{BGH}. 

Now to show (a) $\Rightarrow$ (b), let $E \in \mathcal{AC}(R\op)$ and consider a short exact sequence 
\[
\varepsilon \colon 0 \to E \to I \to E' \to 0
\]
with $I \in \mathcal{I}(R\op)$. Since $\mathcal{AC}(R\op)$ is a coresolving class by \cite[Proposition 2.7 (3)]{BGH}, we have that $E'$ is also absolutely clean. Thus, $\varepsilon$ can be regarded as a complex in $\mathscr{AC}(R\op)$, and since the latter coincides with $\widehat{\mathcal{AC}(R\op)}$, we have that $\varepsilon$ splits. It follows that $E$ is injective, that is, we have proved that $\mathcal{AC}(R\op) = \mathcal{I}(R\op)$. This in turn implies that $R$ is a right Noetherian ring by \cite[Theorem 2.5 and Lemma 5.2]{BP17}.   

Note also that if either (a) or (b) holds, then $\mathcal{AC}(R\op)$ is self-orthogonal, and since the cotorsion pair $({}^\perp(\mathcal{AC}(R\op)),\mathcal{AC}(R\op))$ is cogenerated by a set of representatives of modules of type $\text{FP}_\infty$ (see \cite[Proposition 4.1 and Corollary 4.2]{BP17}), we have that every absolutely clean module in $\Mod(R\op)$ is a direct summand of a module filtered by a set of modules of type $\text{FP}_\infty$ (see \cite[Corollary 3.2.4]{GT}, for instance).

Finally, suppose that any $E \in \mathcal{AC}(R\op)$ satisfies (c). Without loss of generality, we may assume that $E = \varinjlim_{\alpha < \lambda} E_\alpha$ for some ordinal $\lambda > 0$, such that $E_0$ and each quotient $E_{\alpha + 1} / E_\alpha$ (where $E_\alpha \subseteq E_{\alpha + 1}$) are of type $\text{FP}_\infty$. Then, Eklof's Lemma implies that  $\Ext^1_R(E,E') = 0$ for every $E' \in \mathcal{AC}(R\op)$. Hence, the class $\mathcal{AC}(R\op)$ is self-orthogonal, which also implies that $\mathcal{AC}(R\op) = \mathcal{I}(R\op)$, and thus (a). 
\end{remark}

Previously we proved that, in case $\mathcal{GF}_{\mathcal{B}}(R)$ is closed under extensions, then $\Ch(\mathcal{GF}_{\mathcal{B}}(R))$ can be represented as the class $\mathscr{GF}_{\widehat{\mathcal{B}}}(R)$ of Gorenstein $\widehat{\mathcal{B}}$-flat complexes, that is, Gorenstein flat complexes relative to direct sums of disks complexes centred at objects of $\mathcal{B}$. Relative Gorenstein $\otimes^.$-flat complexes provide another description for $\Ch(\mathcal{GF}_{\mathcal{B}}(R))$, by using sphere complexes instead, as specified in the following result.

\begin{proposition}
Let $\mathcal{B}$ be a class of right $R$-modules and assume that $\mathcal{GF}_{\mathcal{B}}(R)$ is closed under extensions. Then, we have the equality
\[
\mathscr{GF}^{\otimes^.}_{\check{\mathcal{B}}}(R) = \Ch(\mathcal{GF}_{\mathcal{B}}(R))
\] 
where
\[
\check{\mathcal{B}} := \left\{ X \in \Ch(R\op) \mbox{ {\rm :} } X \simeq \bigoplus_{m \in \mathbb{Z}} S^m(B_m),\ B_m\in \mathcal{B} \right\}.
\]
\end{proposition}

\begin{proof}
We first prove the inclusion $(\subseteq)$. So let $X \in \mathscr{GF}^{\otimes^.}_{\check{\mathcal{B}}}(R)$, that is, $X = Z_0(F_\bullet)$ where $F_\bullet$ is an exact and $(\check{\mathcal{B}} \otimes^. -)$-acyclic complex of $\otimes^.$-flat complexes. In particular, we have that $S^0(B) \otimes^. F_\bullet$ is an exact complex of complexes of abelian groups, for every $B \in \mathcal{B}$. Let us write
\[
F_\bullet = \cdots \to F_1 \to F_0 \to F_{-1} \to \cdots,
\]
so that 
\[
S^0(B) \otimes^. F_\bullet = \cdots \to S^0(B) \otimes^. F_1 \to S^0(B) \otimes^. F_0 \to S^0(B) \otimes^. F_{-1} \to \cdots,
\]
where 
\[
(S^0(B) \otimes^. F_i)_m = \bigoplus_{k \in \mathbb{Z}} \left( (S^0(B))_k \dson{\otimes}_R F_{i,{m-k}} \right) = B \dson{\otimes}_{R} F_{i, m}.
\]
It follows that we have an exact complex of flat modules
\[
F_{\bullet,m} = \cdots \to F_{1,m} \to F_{0,m} \to F_{-1,m} \to \cdots
\]
which is $(B \otimes_R -)$-acyclic for every $B \in \mathcal{B}$, that is, each 
\[
X_m = {\rm Ker}(F^0 \to F^{-1})_m = Z_0(F_{\bullet,m})
\] 
is a Gorenstein $\mathcal{B}$-flat module. Hence, $X \in \Ch(\mathcal{GF}_{\mathcal{B}}(R))$. 

For the remaining inclusion, consider $X \in \Ch(\mathcal{GF}_{\mathcal{B}}(R))$. Using the same argument as in \cite[Proposition 3.5]{gang:12:gorflatGF} with Gorenstein $\mathcal{B}$-flat modules replacing Gorenstein flat modules, we can construct a flat complex (in particular $\otimes^.$-flat) 
\[
F^0 := \bigoplus_{m \in \mathbb{Z}} D^{m+1}(F_m)
\] 
with $F_m$ flat for every $m \in \mathbb{Z}$, along with a monomorphism $X \rightarrowtail F^0$ such that ${\rm CoKer}(X \rightarrowtail F^0)$ is a complex of Gorenstein $\mathcal{B}$-modules. Repeating this argument infinitely many times, we obtain an exact sequence 
\[
\rho \colon 0 \to X \to F^0 \to F^1 \to \cdots
\]
with cycles $G^i$ in $\Ch(\mathcal{GF}_{\mathcal{B}}(R))$. This sequence $\rho$ is also $(\check{\mathcal{B}} \otimes^. -)$-acyclic, since $\Tor^{\cdot}_k(B,X) = 0$ and $\Tor^{\cdot}_k(B,G^i) = 0$ for every $k, i > 0$ and $B \in \check{\mathcal{B}}$. For, let us write $B \simeq \oplus_{m \in \mathbb{Z}} S^m(B_m)$ with $B_m \in \mathcal{B}$. Then, we have 
\begin{align*}
\Tor^{\cdot}_k(B,X) \cong \bigoplus_{m \in \mathbb{Z}} \Tor^{\cdot}_k(S^m(B_m),X),
\end{align*}
since $\Tor^{\cdot}_k(-,-)$ commutes with coproducts. Also, 
\begin{align*}
(\Tor^{\cdot}_k(S^m(B_m),X))_n \cong \Tor^R_k(B_m,X_{n-m}) = 0
\end{align*}
since $X_{n-m}$ is a Gorenstein $\mathcal{B}$-flat module. Thus, $\Tor^{\cdot}_k(B,X) = 0$, and similarly, $\Tor^{\cdot}_k(B,G^i) = 0$. Hence, $\rho$ is $(\check{\mathcal{B}} \otimes^{\cdot} -)$-acyclic and $\Tor^{\cdot}_k(B,X) = 0$ for every $k > 0$ and $B \in \check{\mathcal{B}}$. Using Lemma \ref{characterisations}, we have that $X$ is Gorenstein $\check{\mathcal{B}}$-flat under $\otimes^.$.   
\end{proof}


\section{The $\mathcal{B}$-flat stable module category.}\label{sec:stable}

If we are given two hereditary complete cotorsion pairs $(\mathcal{Q}, \mathcal{R}')$ and $(\mathcal{Q}',\mathcal{R})$ in an abelian category $\mathcal{C}$ such that $\mathcal{Q}' \subseteq \mathcal{Q}$, $\mathcal{R}' \subseteq \mathcal{R}$ and $\mathcal{Q}' \cap \mathcal{R} = \mathcal{Q} \cap \mathcal{R}'$, then there exists by \cite[Main Theorem 1.2]{GillespieTriple} a unique subcategory $\mathcal{W} \subseteq \mathcal{C}$ such that $(\mathcal{Q,W,R})$ is a Hovey triple in $\mathcal{C}$, that is:
\begin{enumerate}
\item $(\mathcal{Q},\mathcal{R} \cap \mathcal{W})$ and $(\mathcal{Q} \cap \mathcal{W}, \mathcal{R})$ are complete cotorsion pairs in $\mathcal{C}$.

\item $\mathcal{W}$ is \emph{thick}, meaning that it is closed under extensions, kernels of epimorphisms and cokernels of monomorphisms between its objects, and also under direct summands. 
\end{enumerate}
Due to Hovey's correspondence \cite[Theorem 2.2]{Hovey}, the existence of such triple $(\mathcal{Q,W,R})$ implies the existence of a unique abelian model structure on $\mathcal{Q}$ such that:
\begin{itemize}
\item A morphism $f$ is a (trivial) cofibration if, and only if, it is monic and ${\rm CoKer}(f) \in \mathcal{Q}$ (resp., ${\rm CoKer}(f) \in \mathcal{Q} \cap \mathcal{W} = \mathcal{Q}'$).

\item A morphism $g$ is a (trivial) fibration if, and only if, it is epic and ${\rm Ker}(g) \in \mathcal{R}$ (resp., ${\rm Ker}(g) \in \mathcal{R} \cap \mathcal{W} = \mathcal{R}'$).
\end{itemize}


\subsection*{Compatibility relations between flat and relative Gorenstein flat objects}

Now let $\mathcal{B}$ be a class of right $R$-modules that contains the injectives and such that $\mathcal{GF}_{\mathcal B}(R)$ is closed under extensions (for instance if $\mathcal B$ is semi-definable and contains the injectives). We shall show that it is possible to apply the previous result in the setting where:
\begin{align*}
\mathcal{Q} & := \mathcal{GF}_{\mathcal B}(R), & \mathcal{R}' & := \mathcal{GC}_{\mathcal B}(R), & \mathcal{Q}' & := \mathcal{F}(R), & \mathcal{R} & := \mathcal{C}(R).
\end{align*}
The reader can keep in mind the case for which $\mathcal B$ is the class of all injective right $R$-modules (so $\mathcal{GF}_{\mathcal B}(R)$ is the class of Gorenstein flat modules).

It is well known that $(\mathcal{F}(R),\mathcal{C}(R))$ is a hereditary complete cotorsion pair for any ring $R$ (see \cite[Proposition 2]{FlatCoverConjecture} by L. Bican, R. El Bashir and E. E. Enochs). In the setting of complexes, recall from \cite[Corollary 4.10]{GillespieFlat} and \cite[Theorem 4.3]{AEGO} that the classes $\mathscr{F}(R)$ and $\Ch(\mathcal{F}(R))$ of flat complexes and complexes of flat modules are the left halves of two complete cotorsion pairs 
\begin{align*}
(\mathscr{F}(R),\mathscr{C}(R)) & & \text{and} & & (\Ch(\mathcal{F}(R)),(\Ch(\mathcal{F}(R)))^\perp).
\end{align*} 
On the other hand, by Corollary \ref{complete}, 
\begin{align*}
(\mathcal{GF}_{\mathcal B}(R),\mathcal{GC}_{\mathcal B}(R)), & & (\mathscr{GF}_{\mathscr{B}}(R),\mathscr{GC}_{\mathscr{B}}(R)) & & \text{and} & & (\mathscr{GF}_{\mathscr{B}}^{\otimes^.}(R),\mathscr{GC}_{\mathscr{B}}^{\otimes^.}(R))
\end{align*} 
are hereditary perfect cotorsion pairs in $\Mod(R)$ and $\Ch(R)$ in the cases where the left halves are closed under extesions. Therefore, since the inclusions 
\begin{align*}
\mathcal{F}(R) & \subseteq \mathcal{GF}_{\mathcal B}(R) & \text{and} & & \mathcal{GC}_{\mathcal B}(R) & \subseteq \mathcal{C}(R), \\
\mathscr{F}(R) & \subseteq \mathscr{GF}_{\mathscr{B}}(R) & \text{and} & & \mathscr{GC}_{\mathscr{B}}(R) & \subseteq \mathscr{C}(R), \\
\Ch(\mathcal{F}(R)) & \subseteq \mathscr{GF}_{\mathscr{B}}^{\otimes^.}(R) & \text{and} & & \mathscr{GC}_{\mathscr{B}}^{\otimes^.}(R) & \subseteq (\Ch(\mathcal{F}(R)))^\perp
\end{align*} 
are clear, the desired Hovey triples (and thus the associated relative Gorenstein flat model structures on $\Mod(R)$ and $\Ch(R)$) will be a consequence of the following result.

\begin{proposition}[compatibility between flat and relative Gorenstein flat objects]\label{prop:comp_condition}
The following equalities hold true for classes $\mathcal{B} \subseteq \mathsf{Mod}(R^{\rm op})$ and $\mathscr{B} \subseteq \mathsf{Ch}(R^{\rm op})$ containing $\mathcal{I}(R^{\rm op})$ and $\mathscr{I}(R^{\rm op})$, respectively.
\begin{enumerate}
\item $\mathcal{F}(R) \cap \mathcal{C}(R) = \mathcal{GF}_{\mathcal B}(R) \cap \mathcal{GC}_{\mathcal B}(R)$, provided that $\mathcal{GF}_{\mathcal B}(R)$ is closed under extensions.

\item $\mathscr{F}(R) \cap \mathscr{C}(R) = \mathscr{GF}_{\mathscr{B}}(R) \cap \mathscr{GC}_{\mathscr{B}}(R)$, provided that $\mathscr{GF}_{\mathscr B}(R)$ is closed under extensions.

\item $\Ch(\mathcal{F}(R)) \cap (\Ch(\mathcal{F}(R)))^\perp = \mathscr{GF}_{\mathscr{B}}^{\otimes^.}(R) \cap \mathscr{GC}_{\mathscr{B}}^{\otimes^.}(R)$, provided that $\mathscr{GF}_{\mathscr{B}}^{\otimes^.}(R)$ is closed under extensions.
\end{enumerate}
\end{proposition}

\begin{proof} \
\begin{enumerate}
\item Let us first prove the containment $(\supseteq)$. So suppose we are given a module $M \in \mathcal{GF}_{\mathcal{B}}(R) \cap \mathcal{GC}_{\mathcal{B}}(R)$. We already have that $M \in \mathcal{C}(R)$. On the other hand, since $M$ is Gorenstein $\mathcal{B}$-flat, we have by Lemma \ref{characterisations} a short exact sequence 
\[
0 \to M \to F \to M' \to 0
\] 
where $F$ is flat and $M'$ is Gorenstein $\mathcal{B}$-flat. This sequence splits, since $M$ is Gorenstein $\mathcal B$-cotorsion and so $\Ext^1_R(M',M) = 0$. Hence, $M$ is a direct summand of the flat module $F$, and so $M \in \mathcal{F}(R)$.

Now let us show the remaining containment $(\subseteq)$. Let $N \in \mathcal{F}(R) \cap \mathcal{C}(R)$. Then, it is clear that $N$ is Gorenstein $\mathcal{B}$-flat. On the other hand, since $(\mathcal{GF}_{\mathcal{B}}(R),\mathcal{GC}_{\mathcal{B}}(R))$ is a complete cotorsion pair, there exists a short exact sequence 
\[
0 \to N \to C \to F \to 0
\]
where $C \in \mathcal{GC}_{\mathcal{B}}(R)$ and $F \in \mathcal{GF}_{\mathcal{B}}(R)$. Since $N$ and $F$ are Gorenstein $\mathcal{B}$-flat and $\mathcal{GF}_{\mathcal{B}}(R)$ is closed under extensions, we have that $C \in \mathcal{GF}_{\mathcal{B}}(R) \cap \mathcal{GC}_{\mathcal{B}}(R) \subseteq \mathcal{F}(R) \cap \mathcal{C}(R)$. It follows that $F$ is a Gorenstein flat module with finite flat dimension, and so $F$ is flat by \cite[Corollary 10.3.4]{EJ}. Then, we have that $\Ext^1_R(F,N) = 0$ since $N$ is cotorsion, and so the previous exact sequence splits. It follows that $N$ is a direct summand of $C \in \mathcal{GC}_{\mathcal{B}}(R)$, and hence $N \in \mathcal{GC}_{\mathcal{B}}(R)$.

\item The containment $(\supseteq)$ follows as in the module case. Now let $X \in \mathscr{F}(R) \cap \mathscr{C}(R)$. As in (1), we can obtain a short exact sequence 
\begin{align}\label{sequence}
0 & \to X \to C \to F \to 0
\end{align}
where $F \in \mathscr{GF}_{\mathscr{B}}(R)$ and $C \in \mathscr{GF}_{\mathscr{B}}(R) \cap \mathscr{GC}_{\mathscr{B}}(R) \subseteq \mathscr{F}(R) \cap \mathscr{C}(R)$. It follows that $F$ is a Gorenstein flat complex (since $\mathscr{B}$ contains the injectives) with finite flat dimension $\leq 1$. Then, in particular, each $F_m$ is a Gorenstein flat module and so $F_m^+$ is Gorenstein injective. Thus $F^+$ is a Gorenstein injective complex (see \cite[Proposition 2.8]{yang:11:gorflat}). Also, ${\rm id}(F^+) = {\rm fd}(F) \leq 1$. So there is an exact sequence of complexes 
\[
0 \rightarrow F^+ \rightarrow E^0 \rightarrow E^1\rightarrow 0,
\] 
with $E^0$ and $E^1$ injective complexes. Since $E^1$ is injective and $F^+$ is Gorenstein injective, the sequence splits, so $F^+$ is an injective complex. It follows that $F$ is a flat complex (see \cite[Theorem 4.1.3]{jrgr}). Recall also that $X \in \mathscr{C}(R)$, and so the sequence~\eqref{sequence} splits. Therefore $X \in \mathscr{GC}_{\mathscr{B}}(R)$.

\item The containment $(\supseteq)$ follows as in the module case. Now for any $F \in \Ch(\mathcal{F}(R)) \cap (\Ch(\mathcal{F}(R)))^\perp$, we can proceed as in (1) to obtain a short exact sequence 
\begin{align}\label{eqn:FCL}
0 \to F \to C \to L \to 0
\end{align}
with $L \in \mathscr{GF}_{\mathscr{B}}^{\otimes^.}(R)$ and $C \in \mathscr{GF}_{\mathscr{B}}^{\otimes^.}(R) \cap \mathscr{GC}^{\otimes^.}_{\mathscr{B}}(R) \subseteq \Ch(\mathcal{F}(R)) \cap (\Ch(\mathcal{F}(R)))^\perp$. For the rest of the proof, we shall focus of showing that $L \in \Ch(\mathcal{F}(R))$. 

We know that $L \in \mathscr{GF}_{\mathscr{B}}^{\otimes^.}(R)$. By Proposition \ref{prop:relative_GF_is_GF} we have that $L \in \mathscr{GF}(R) = \Ch(\mathcal{GF}(R))$. On the other hand, $F, C \in \Ch(\mathcal{F}(R))$. Thus, for each $m \in \mathbb{Z}$, we have a short exact sequence 
\[
0 \to F_m \to C_m \to L_m \to 0
\] 
where $F_m$ and $C_m$ are flat modules and $L_m$ is Gorenstein flat, that is, $L_m$ is a Gorenstein flat module with flat dimension $\leq 1$. This implies that $L_m$ is flat for every $m \in \mathbb{Z}$, that is, $L \in \Ch(\mathcal{F}(R))$.
\end{enumerate}
\end{proof}


\subsection*{Relative Gorenstein flat model structures}

From the previous section, we can deduce the following result.

\begin{theorem}[the relative Gorenstein flat model structures on $\Mod(R)$ and $\mathsf{Ch}(R)$]\label{theo:GBF_model}
Let $\mathcal{B} \subseteq \mathsf{Mod}(R)$ and $\mathscr{B} \subseteq \mathsf{Ch}(R)$ be classes containing $\mathcal{I}(R^{\rm op})$ and $\mathscr{I}(R^{\rm op})$, respectively. Then, the following assertions hold:
\begin{enumerate}
\item \underline{The Gorenstein $\mathcal{B}$-flat model structure}: If $\mathcal{GF}_{\mathcal B}(R)$ is closed under extensions, then there exists a unique abelian model structure on $\Mod(R)$ such that $\mathcal{GF}_{\mathcal{B}}(R)$ is the class of cofibrant objects, and $\mathcal{C}(R)$ is the class of fibrant objects. In the case $\mathcal{B}$ is semi-definable, the class of trivial objects is given by $(\mathcal{PGF}_{\mathcal{B}}(R))^\perp$ (see Theorem \ref{theo:cotorsion_pair_PGFB}).

\item \underline{The Gorenstein $\mathscr{B}$-flat model structure}: If $\mathscr{GF}_{\mathscr B}(R)$ is closed under extensions, then there exists a unique abelian model structure on $\Ch(R)$ such that $\mathscr{GF}_{\mathscr{B}}(R)$ is the class of cofibrant objects, and $\mathscr{C}(R)$ is the class of fibrant objects. In the case $\mathscr{B}$ is semi-definable, the class of trivial objects is given by $(\mathscr{PGF}_{\mathscr{B}}(R))^\perp$.

\item \underline{The Gorenstein $\mathscr{B}$-flat model structure}: If $\mathscr{GF}_{\mathscr{B}}^{\otimes^.}(R)$ is closed under extensions, then there exists a unique abelian model structure on $\Ch(R)$ with $\mathscr{GF}_{\mathscr{B}}^{\otimes^.}(R)$ as the class of cofibrant objects, and whose fibrant objects are given by $(\Ch(\mathcal{F}(R)))^\perp$. 
\end{enumerate}
\end{theorem}

\begin{proof}
We only show how to obtain the Gorenstein $\mathcal{B}$-flat model structure on $\mathsf{Mod}(R)$, as the other two structures can be obtained in a similar way.

The first part is clear due to the compatibility relation proved in part (1) of Propositon \ref{prop:comp_condition}. Now let us assume that $\mathcal{B}$ is also semi-definable. By the comments at the beginning of this section, there exists a unique thick class $\mathcal{W}$ modules (the class of trivial objects) satisfying 
\begin{align*}
\mathcal{GF}_{\mathcal{B}}(R) \cap \mathcal{W} & = \mathcal{F}(R) & \text{ and } & & \mathcal{C}(R) \cap \mathcal{W} & = \mathcal{GC}_{\mathcal{B}}(R).
\end{align*}
Now, by Theorem \ref{theo:equivalences_GF} (conditions (c) and (d)), the class $(\mathcal{PGF}_{\mathcal{B}}(R))^\perp$ satisfies 
\[
\mathcal{GF}_{\mathcal{B}}(R)\cap (\mathcal{PGF}_{\mathcal{B}}(R))^\perp = \mathcal{F}(R),
\]
and 
\[
\mathcal{C}(R)\cap (\mathcal{PGF}_{\mathcal{B}}(R))^\perp \subseteq \mathcal{GC}_{\mathcal{B}}(R). 
\] 
But, since $\mathcal{PGF}_{\mathcal{B}}(R)\subseteq \mathcal{GF}_{\mathcal{B}}(R)$ and $\mathcal{GC}_{\mathcal{B}}(R)\subseteq \mathcal{C}(R)$, the containment 
\[
\mathcal{C}(R)\cap (\mathcal{PGF}_{\mathcal{B}}(R))^\perp \supseteq \mathcal{GC}_{\mathcal{B}}(R)
\] 
also holds. The class $(\mathcal{PGF}_{\mathcal{B}}(R))^\perp$ is clearly closed under extensions and direct summands, and by Theorem \ref{theo:cotorsion_pair_PGFB} it is also closed under cokernels of momorphisms. Let us finally see that $(\mathcal{PGF}_{\mathcal{B}}(R))^\perp$ is closed under kernels of epimorphisms. To this aim, let 
\[
0\to A\to B\to C\to 0
\] 
be a short exact sequence with $B,C\in (\mathcal{PGF}_{\mathcal{B}}(R))^\perp$. Firstly, since $(\mathcal{PGF}_{\mathcal{B}}(R),(\mathcal{PGF}_{\mathcal{B}}(R))^\perp)$ is a hereditary cotorsion pair (Theorem \ref{theo:cotorsion_pair_PGFB}), we immediately follow from the long exact sequence of cohomology associated to the previous short exact sequence that $\Ext^2_R(L,A)=0$, for each $L\in \mathcal{PGF}_{\mathcal{B}}(R)$. Now fix $F\in \mathcal{PGF}_{\mathcal{B}}(R)$. Then there exists a short exact sequence 
\[
0\to F\to P\to K\to 0,
\] 
with $K\in \mathcal{PGF}_{\mathcal{B}}(R)$ and $P$ projective. Then we have an exact sequence 
\[
\Ext^1_R(P,A)\to \Ext^1_R(F,A)\to \Ext^2_R(K,A).
\] 
Since the two terms in the extremes are 0, the middle one is also 0, that is $A\in (\mathcal{PGF}_{\mathcal{B}}(R))^\perp$, and so $(\mathcal{PGF}_{\mathcal{B}}(R))^\perp$ is closed under kernels of epimorphisms. Hence, it is a thick subcategory. Therefore, by the uniqueness of the class of trivial objects in a Hovey triple, we follow that $\mathcal {W} = (\mathcal{PGF}_{\mathcal{B}}(R))^\perp$.
\end{proof}

\begin{remark}\label{rem:scheme}
Let us make some comments concerning the cotorsion pair $(\mathcal{GF}_{\mathcal{B}}(R),\mathcal{GC}_{\mathcal{B}}(R))$. Under the assumption that $\mathcal{GF}_{\mathcal{B}}(R)$ is closed under extensions ($\mathcal{B}$ is not necessarily semi-definable), we can show that $(\mathcal{GF}_{\mathcal{B}}(R),\mathcal{GC}_{\mathcal{B}}(R))$ is a complete cotorsion pair such that 
\[
\mathcal{GF}_{\mathcal B}(R) \cap \mathcal{GC}_{\mathcal B}(R) = \mathcal{F}(R) \cap \mathcal{C}(R)
\] 
without using the fact that $\Mod(R)$ has enough projective modules. So the Gorenstein $\mathcal{B}$-flat model structure could be obtained in settings more general than modules or chain complexes. A good question in this sense is what conditions we need on a scheme $X$ in order to obtain a Gorenstein flat model structure on the category $\mathfrak{Qcoh}(X)$ of quasi-coherent sheaves over $X$ (which does not have enough projectives).\footnote{This is solved for semi-separeted noetherian schemes in \cite[Theorem 2.5]{CET}.} 

Here lies a difference with respect to {\v{S}}aroch and {{\v{S}{t'o}}}v{\'{\i}}{\v{c}}ek's work. In \cite[Corollary 3.11]{SarochStovicek}, they proved that $(\mathcal{GF}(R),\mathcal{C}(R) \cap (\mathcal{PGF}(R))^\perp)$ is a complete cotorsion pair. This is based on the construction of a cotorsion pair $(\mathcal{PGF}(R),(\mathcal{PGF}(R))^\perp)$ formed by the projectively coresolved Gorenstein flat modules, which uses the existence of enough projectives in $\Mod(R)$.
\end{remark}

The previous theorem yields the following particular model structures on $\mathsf{Mod}(R)$ and $\mathsf{Ch}(R)$.

\begin{corollary}[the Gorenstein flat and Gorenstein AC-flat model structures]
Let $R$ be an arbitrary associative ring with identity. Then, the following assertions hold: 
\begin{enumerate} 
\item \underline{The Gorenstein flat model structure on $\mathsf{Mod}(R)$}: There exists a unique abelian model structure on $\Mod(R)$ such that $\mathcal{GF}(R)$ is the class of cofibrant objects and $\mathcal{C}(R)$ is the class of fibrant objects.\footnote{Similarly, there is a unique abelian model structure on $\mathsf{Ch}(R)$ where $\mathscr{GF}(R)$ is the class of cofibrant objects and $\mathscr{C}(R)$ is the class of fibrant objects.} 

\item \underline{The Gorenstein AC-flat model structure on $\mathsf{Mod}(R)$}: There exists a unique abelian model structure on $\Mod(R)$ with the same fibrant objects and whose cofibrant objects are the Gorenstein AC-flat modules.\footnote{Similarly, there is a unique abelian model structure on $\mathsf{Ch}(R)$ such that the Gorenstein AC-flat complexes are the cofibrant objects, and whose fibrant objects are the cotorsion complexes.} 
\end{enumerate}
\end{corollary}


\subsection*{Stable categories of cotorsion relative Gorenstein flat objects}

The Gorenstein flat model structure just mentioned was first found in \cite{SarochStovicek} for arbitrary rings, although it had been found previously for particular choices of $R$, such as Gorenstein rings \cite{GillespieHovey}, Ding-Chen rings \cite{GillespieDCh} or coherent rings \cite{GillespieGF}. 

The Gorenstein AC-flat model structure, on the other hand, was previously unknown. Following Gillespie's arguments for \cite[Corollary 3.4]{GillespieGF}, we can note that the class $\mathcal{GF}_{\rm AC}(R) \cap \mathcal{C}(R)$ of cotorsion Gorenstein AC-flat modules is a Frobenius category (that is, an exact category with enough projectives and injectives, and in which the projective and injective objects coincide). Thus, we have a stable category 
\[
[\mathcal{GF}_{\rm AC}(R) \cap \mathcal{C}(R)] / \sim,
\] 
where $\sim$ is the equivalence relation defined for morphisms between modules in $\mathcal{GF}_{\rm AC}(R) \cap \mathcal{C}(R)$ given by $f \sim g$ if, and only if, the difference $f - g$ factors through a flat cotorsion module (that is, the projective-injective objects of $\mathcal{GF}_{\rm AC}(R) \cap \mathcal{C}(R)$). Moreover, since the cotorsion pairs 
\begin{align*}
(\mathcal{GF}_{\rm AC}(R),(\mathcal{GF}_{\rm AC}(R))^\perp) & & \text{and} & & (\mathcal{F}(R),\mathcal{C}(R))
\end{align*}
that induce the Gorenstein AC-flat model structure are hereditary, we have that this model is hereditary in the sense of \cite{GillespieExact}. The same reasoning can be applied to any class of Gorenstein flat modules relative to a semi-definable class. Hence, the following result can be obtained as \cite[Corollary 3.4]{GillespieGF}.

\begin{corollary}\label{coro:relative_stable_category}
Let $\mathcal{B}$ be a class of right $R$-modules that contains the injectives and such that $\mathcal{GF}_{\mathcal B}(R)$ is closed under extensions. Then, the class $\mathcal{GF}_{\mathcal{B}}(R) \cap \mathcal{C}(R)$ of Gorenstein $\mathcal{B}$-flat cotorsion modules is a Frobenius category with the exact structure given by the short exact sequences with terms in $\mathcal{GF}_{\mathcal{B}}(R) \cap \mathcal{C}(R)$. The projective-injective modules are given by the flat-cotorsion modules. Furthermore, the homotopy category of the Gorenstein $\mathcal{B}$-flat model structure from Theorem~\ref{theo:GBF_model} is triangle equivalent to the stable category 
\[
[\mathcal{GF}_{\mathcal{B}}(R) \cap \mathcal{C}(R)] / \sim,
\] 
with the relation $\sim$ defined above. 
\end{corollary}

\begin{remark}
We can also have an statement similar to Corollary~\ref{coro:relative_stable_category} is the context of complexes, that is, there is a stable category 
\[
[\mathscr{GF}_{\mathscr{B}}(R) \cap \mathscr{C}(R)] / \sim
\] 
of cotorsion relative Gorenstein flat complexes, where $f \sim g$ if $f-g$ factors through a flat cotorsion complex, which is triangle equivalent to the homotopy category of the Gorenstein $\mathscr{B}$-flat model structure. 

We have a similar description for the stable category 
\[
[\mathscr{GF}^{\otimes^.}_{\mathscr{B}}(R) \cap (\Ch(\mathcal{F}(R)))^\perp] / \sim
\]
in terms of the model structure involving $\mathscr{GF}^{\otimes^.}_{\mathscr{B}}(R)$. Using the result \cite[Theorem 4.10]{GillespieModels} by Gillespie, explained in more detailed in the next section, note that the class $\Ch(\mathcal{F}(R)) \cap (\Ch(\mathcal{F}(R)))^\perp$ of the projective-injective objects of the Frobenius category $\mathscr{GF}^{\otimes^.}_{\mathscr{B}}(R) \cap (\Ch(\mathcal{F}(R)))^\perp$ is the class of contractible complexes of flat cotorsion modules. Moreover, due to the description of the class $(\Ch(\mathcal{F}(R)))^\perp$ given in \cite[Proposition 3.2]{GillespieDegreewise}, we have that every projective-injective complex $P$ is exact and each $P_m$ is a flat-cotorsion module. It then follows by \cite[Theorem 4.1]{BCE} that $P$ has cotorsion cycles, that is, $P \in \Ch(\mathcal{F}(R)) \cap \widetilde{\mathcal{C}(R)}$, where $\widetilde{\mathcal{C}(R)}$ denotes the class of exact complexes with cotorsion cycles. 
\end{remark}

In the following section, we shall study the relation between the Gorenstein $\mathscr{B}$-flat model structure and model structures constructed from relative Gorenstein flat modules. This will allow us to give other descriptions of the stable category $[\mathscr{GF}_{\mathscr{B}}(R) \cap \mathscr{C}(R)] / \sim$.


\section{\textbf{Model structures arising from relative Gorenstein flat modules}}\label{sec:comparison}

In this section we shall make use of a recent result by Gillespie \cite[Theorem 4.10]{GillespieModels} to yield a recollement in $\Ch(R)$ between homotopy categories that involve the class of Gorenstein $\mathcal{B}$-flat modules. To this aim, we shall assume that the class $\mathcal{GF}_{\mathcal{B}}(R)$ is closed under extensions, like for instance when we take $\mathcal{B}$ as a semi-definable class of right $R$-modules.

\begin{theorem}\label{inducedmodels}
Assume that $\mathcal{GF}_{\mathcal B}(R)$ is closed under extensions. Then we have three hereditary abelian model structures given by the triples:
\begin{align*}
\mathcal{M}_1 & = ({\rm ex}(\mathcal{GF}_{\mathcal{B}}(R)),\mathcal{W}_1, {\rm dg}(\mathcal{GC}_{\mathcal{B}}(R))), \\ 
\mathcal{M}_2 & = (\Ch(\mathcal{GF}_{\mathcal B}(R)), \mathcal{W}_2, {\rm dg}(\mathcal{GC}_{\mathcal{B}}(R))), \\
\mathcal{M}_3 & = ({\rm dg}(\mathcal{GF}_{\mathcal{B}}(R)), \mathcal{E}, {\rm dg}(\mathcal{GC}_{\mathcal{B}}(R))).
\end{align*}
The core of each triple $\mathcal{M}_1, \mathcal{M}_2$ and $\mathcal{M}_3$ equals to the class of contractible complexes with components in $\mathcal{GF}_{\mathcal{B}}(R) \cap \mathcal{GC}_{\mathcal{B}}(R)$. So we have a left recollement between the corresponding homotopy categories:
\[
\xymatrix{{\rm Ho}(\mathcal M_1)
 \ar[rr]|j
&&  {\rm Ho}(\mathcal M_2) \ar@/_1pc/[ll]\ar@/^1pc/[ll]\ar[rr]|w && {\rm Ho}(\mathcal M_3) \ar@/_1pc/[ll]\ar@/^1pc/[ll] },\]

\bigskip\par\noindent
where 
\[
{\rm Ho}(\mathcal{M}_1) \cong \frac{\mathsf{K}_{\mathrm{ac}}(\mathcal{GF}_{\mathcal{B}}(R))}{\widetilde{\mathcal{GF}}_{\mathcal{B}}(R)}, \mbox{ } {\rm Ho}(\mathcal{M}_2) \cong \frac{\mathsf{K}(\mathcal{GF}_{\mathcal{B}}(R))}{\widetilde{\mathcal{GF}}_{\mathcal{B}}(R)} \mbox{ \ and \ } {\rm Ho}(\mathcal{M}_3) \cong \mathsf{D}(R).
\]
Here:
\begin{itemize} 
\item $\mathcal{E}$ denotes the class of exact complexes in $\Ch(R)$, and ${\rm ex}(\mathcal{GF}_{\mathcal{B}}(R)) = \Ch(\mathcal{GF}_{\mathcal{B}}(R)) \cap \mathcal{E}$.

\item ${\rm dg}(\mathcal{GC}_{\mathcal{B}}(R))$ is the class of complexes $Y$ in $\Ch(\mathcal{GC}_{\mathcal{B}}(R))$ such that $\mathcal{H}{\rm om}(X,Y)$ is exact whenever $X$ is an exact complex with cycles in $\mathcal{GF}_{\mathcal{B}}(R)$. The class of such complexes in $\mathsf{K}(\mathcal{GF}_{\mathcal{B}}(R))$ is denoted by $\widetilde{\mathcal{GF}}_{\mathcal{B}}(R)$.
\end{itemize}
\end{theorem}


\subsection*{\textbf{Comparison between Gorenstein flat models in $\Ch(R)$}}

Let $\mathcal{B}$ be a class of right $R$-modules, and $\mathscr{D}$ a class of complexes such that $\mathscr{D} \supseteq \Ch(\mathcal{B})$. Assume that $\mathscr{D}$ is semi-definable and contains all injective complexes. From Theorems \ref{theo:GBF_model} and \ref{inducedmodels} we have the two models $\Ch(R)_{\mathscr{D}\textrm{-flat}}$ and $\mathcal M_2$ on $\Ch(R)$. They are given by the triples 
\[
\Ch(R)_{\mathscr{D}\textrm{-flat}} = (\mathscr{GF}_{\mathscr{D}}(R), \mathcal{W}, {\rm dg}(\mathcal{C}(R)) \mbox{ \ and \ } \mathcal{M}_2 = (\Ch(\mathcal{GF}_{\mathcal{B}}(R)),\mathcal{W}_2, {\rm dg}(\mathcal{GC}_{\mathcal{B}}(R))).
\] 
The homotopy category  ${\rm Ho}(\Ch(R)_{\mathscr{D}\textrm{-flat}})$ is triangle equivalent to the stable category 
\[
[\mathscr{GF}_{\mathscr{D}}(R) \cap {\rm dg}(\mathcal{C}(R))] / \sim,
\] 
where $f\sim g$ if $f-g$ factors through a complex in $\mathscr{F}(R) \cap \mathscr{C}(R)$ (that is, a flat-cotorsion complex). In turn, the homotopy category ${\rm Ho}(\mathcal{M}_2)$ is  triangle equivalent to the derived category
\[
\mathsf{D}(\mathcal{GF}_{\mathcal{B}}(R)) := \frac{\mathsf{K}(\mathcal{GF}_{\mathcal{B}}(R))}{\widetilde{\mathcal{GF}}_{\mathcal{B}}(R)}.
\]
In this section we get an adjunction between these two homotopy categories. We need to recall the notion of Quillen adjunction between two model categories:

\begin{definition}
Suppose $\mathcal M$ and $\mathcal{M'}$ are model categories.
\begin{enumerate}
\item We call a functor $F: \mathcal M \rightarrow \mathcal{M'}$ a \emph{left Quillen functor} if $F$ is a left adjoint and preserves cofibrations and trivial cofibrations.

\item We call a functor $U: \mathcal{M'} \rightarrow \mathcal M$ a \emph{right Quillen functor} if $U$ is a right adjoint and preserves fibrations and trivial fibrations.

\item Suppose $(F, U, \varphi)$ is an adjunction from $\mathcal M$ to $\mathcal{M'}$. That is, $F$ is a functor $\mathcal M \rightarrow \mathcal{M'}$, $U$ is a functor $\mathcal{M'} \rightarrow \mathcal M$, and $\varphi$ is a natural isomorphism $\Hom(FA,B) \rightarrow \Hom(A,UB)$ expressing $U$ as a right adjoint of $F$. We call $(F, U, \varphi)$ a \emph{Quillen adjunction} if $F$ is a left Quillen functor.
\end{enumerate}
\end{definition}

\begin{lemma}\label{aQ1}\cite[Lemma 1.3.4]{hovey2}
Suppose $(F,U,\varphi): \mathcal M \rightarrow \mathcal{M'}$ is an adjunction, and $\mathcal M$ and $\mathcal{M'}$ are model categories. Then $(F,U,\varphi)$ is a Quillen adjunction if and only if $U$ is a right Quillen functor.
\end{lemma}

\begin{definition}
Suppose $\mathcal{M}$ and $\mathcal{M'}$ are model categories.
\begin{enumerate}
\item If $F : \mathcal M \rightarrow \mathcal{M'}$ is a left Quillen functor, define the total left derived functor
$LF: {\rm Ho} (\mathcal M) \rightarrow {\rm Ho} (\mathcal{M'})$ to be the composite $$\xymatrix{{\rm Ho} (\mathcal M) \ar[r]^{{\rm Ho}(Q)} & {\rm Ho} (\mathcal M_c) \ar[r]^{{\rm Ho}(F)} & {\rm Ho} (\mathcal{M'}}),$$
where $Q$ is the cofibrant replacement functor.
Given a natural transformation $\tau : F \rightarrow F' $ of left Quillen functors, define the total derived natural transformation $L_{\tau}$ to be $ {\rm Ho} (\tau) \circ {\rm Ho} (Q)$, so that $(L \tau)_X = \tau_{QX}$.

\item  If $U : \mathcal{M'} \rightarrow \mathcal M$ is a right Quillen functor, define the total right derived functor $RU: {\rm Ho} (\mathcal{M'}) \rightarrow {\rm Ho} (\mathcal M)$ of $U$ to be the composite $$\xymatrix{ {\rm Ho} (\mathcal{M'}) \ar[r]^{{\rm Ho}(R)} &  {\rm Ho} (\mathcal{M'}_f) \ar[r]^{{\rm Ho}(U)} & {\rm Ho} (\mathcal M}),$$ where $R$ is the fibrant replacement functor.
Given a natural transformation $\tau : U \rightarrow U'$ of right Quillen functors, define the total derived natural transformation $R \tau$ to be ${\rm Ho} (\tau) \circ {\rm Ho} (R)$, so that $R\tau_X = \tau_{RX} X$.
\end{enumerate}
\end{definition}

\begin{lemma}\label{aQ2}\cite[Lemma 1.3.10]{hovey2}
Suppose $\mathcal M$ and $\mathcal{M'}$ are model categories and $(F,U,\varphi): \mathcal M \rightarrow  \mathcal{M'}$ is a Quillen adjunction. Then $LF$ and $RU$ are part of an adjunction $L(F,U,\varphi) =
(LF,RU,R\varphi)$, which we call the \emph{derived adjunction}.
\end{lemma}

\begin{proposition}\label{prop-adjunction}
Let us consider the models $\Ch(R)_{\mathscr{D}\textrm{-}{\rm flat}}$ and $\mathcal{M}_2$ on $\Ch(R)$ given by the Hovey triples 
\begin{align*}
\Ch(R)_{\mathscr{D}\textrm{-}{\rm flat}} & = (\mathscr{GF}_{\mathscr{D}}(R),\mathcal{W}, {\rm dg}(\mathcal{C}(R))) & \text{and} & & \mathcal{M}_2 & = (\Ch(\mathcal{GF}_{\mathcal{B}}(R)),\mathcal{W}_2,{\rm dg}(\mathcal{GC}_{\mathcal{B}}(R))).
\end{align*}
Then ${\rm id}: \Ch(R)_{\mathscr{D}\textrm{-}{\rm flat}} \rightarrow \mathcal M_2 $ is a left Quillen functor. So there is a derived adjunction between 
\begin{align*}
[\mathscr{GF}_{\mathscr{D}}(R) \cap {\rm dg}(\mathcal{C}(R))] / \sim & & \text{and} & & \mathsf{D}(\mathcal{GF}_{\mathcal{B}}(R)).
\end{align*}
\end{proposition}

\begin{proof}
First of all, it is clear that the functor ${\rm id} \colon \mathcal{M}_2 \rightarrow \Ch(R)_{\mathscr{D}\textrm{-flat}}$ is a right adjoint functor of ${\rm id} \colon \Ch(R)_{\mathscr{D}\textrm{-flat}} \rightarrow \mathcal{M}_2$. To prove that ${\rm id} \colon \Ch(R)_{\mathscr{D}\textrm{-flat}} \rightarrow \mathcal{M}_2$ is a left Quillen functor, we need to show that a cofibration (resp. a trivial cofibration) in $\Ch(R)_{\mathscr{D}\textrm{-flat}}$ is also a cofibration in $\mathcal{M}_2$ (resp. a trivial cofibration in $\mathcal{M}_2$). 

Let us first show the claim for cofibrant maps. A cofibration in the model $\Ch(R)_{\mathscr{D}\textrm{-flat}}$ is a monomorphism with cokernel a Gorenstein $\mathscr{D}$-flat complex (that is, a complex in $\mathscr{GF}_{\mathscr{D}}(R)$), and a cofibration in the model $\mathcal{M}_2$ is a monomorphism with cokernel a complex of Gorenstein $\mathcal{B}$-flat modules (i.e. a complex in $\Ch(\mathcal{GF}_{\mathcal{B}}(R))$). By Lemma \ref{relation.complexes} we have the containment $\mathscr{GF}_{\mathscr{D}}(R) \subseteq \Ch(\mathcal{GF}_{\mathcal{B}}(R))$. Hence the claim follows. 

Let us see the case of trivial cofibrant maps. But a trivial cofibration in the model $\Ch(R)_{\mathscr{D}\textrm{-flat}}$ is a monomorphism with cokernel a flat complex and a trivial cofibration in the model $\mathcal{M}_2$ is a monomorphism with cokernel in $\widetilde{\mathcal{GF}}(R)$, so the statement follows since $\mathscr{F}(R) \subseteq \widetilde{\mathcal{GF}}_{\mathcal{B}}(R)$.

So  ${\rm id} \colon \Ch(R)_{\mathscr{D}\textrm{-flat}} \rightarrow \mathcal{M}_2$ is a left Quillen functor, and then by Lemma \ref{aQ1}, we have that ${\rm id} \colon \mathcal{M}_2\rightarrow \Ch(R)_{\mathscr{D}\textrm{-flat}}$ is a right Quillen functor. Finally from Lemma~\ref{aQ2} we have a derived adjunction $(L({\rm id}),R({\rm id}))$ given by the total left derived functor 
\[
L({\rm id}) \colon [\mathscr{GF}_{\mathscr{D}}(R) \cap {\rm dg}(\mathcal{C}(R))] / \!\!\sim\, \rightarrow  \mathsf{D}(\mathcal{GF}_{\mathcal{B}}(R))
\] 
and the total right derived functor 
\[
R({\rm id}) \colon \mathsf{D}(\mathcal{GF}_{\mathcal{B}}(R)) \rightarrow [\mathscr{GF}_{\mathscr{D}}(R) \cap {\rm dg}(\mathcal{C}(R))] / \!\!\sim.
\]
\end{proof}


\subsection*{\textbf{Comparison between the $\mathcal B$-flat model and the induced degreewise flat model in $\Ch(R)$}}

From now on we shall consider a class $\mathcal B$ of right $R$-modules in the assumptions of Proposition \ref{dwGBflat.complexes}. Then, as in Proposition \ref{dwGBflat.complexes}, $\widehat{\mathcal B}$ will denote its associated class of complexes. For example, we can think in $\mathcal B$ as the class of injective right $R$-modules, so then $\widehat{\mathcal B}$ coincides with the class of injective complexes.

The flat cotorsion pair $(\mathcal F(R),\mathcal C(R))$ in $\Mod(R)$ induces a degreewise model structure in $\Ch(R)$, denoted by  $\Ch(R)_{{\rm dw}\textrm{-}{\rm flat}} $, and given by the triple: 
\[
\Ch(R)_{{\rm dw}\textrm{-}{\rm flat}}=(\Ch(\mathcal{F}(R)), \mathcal{V}, {\rm dg}(\mathcal{C}(R))).
\] 
Its homotopy category is the derived category $\mathsf D(\mathcal{F}(R))$ of flat complexes and it is triangle equivalent to the Verdier quotient
\[
\frac{\mathsf{K}(\mathcal{F}(R))}{\widetilde{\mathcal{F}}(R)}.
\] (here $\widetilde{\mathcal{F}}(R)$ are the flat complexes as a localizing subcategory of $\mathsf{K}(\mathcal{F}(R))$).

Under the assumptions of Proposition \ref{dwGBflat.complexes} we get that $\mathscr{GF}_{\widehat{\mathcal B}}=\Ch(\mathcal{GF}_{{\mathcal B}}(R))$. Therefore in this case Theorem \ref{theo:GBF_model} gives the following model in $\Ch(R)$: 
\[
\Ch(R)_{\widehat{\mathcal{B}}\textrm{-flat}} = (\Ch(\mathcal{GF}_{{\mathcal{B}}}(R)),\mathcal{W}, {\rm dg}(\mathcal{C}(R))).
\] 
Since the containment $\Ch(\mathcal F(R))\subseteq \Ch(\mathcal{GF}_{{\mathcal B}}(R))$ always holds, and the trivially cofibrant objects of the two models agree (the class of flat complexes), the same argument of Proposition \ref{prop-adjunction} applies to show that the identity functor is a Quillen adjunction between the two models:

\begin{proposition}
The identity functor ${\rm id}:\Ch(R)_{{\rm dw}\textrm{-}{\rm flat}}  \rightarrow \Ch(R)_{\widehat{\mathcal{B}}\textrm{-flat}} $ is a left Quillen functor. So there is a derived adjunction between 
\begin{align*}
\mathsf D(\mathcal{F}(R)) & & \text{and} & & [\mathscr{GF}_{\widehat{\mathcal B}}(R) \cap {\rm dg}(\mathcal{C}(R))] / \sim.
\end{align*}
\end{proposition}


\appendix
\section{The double dual of a definable class of chain complexes}

This appendix concerns to some comments on the proof of Lemma~\ref{lem:Prest}. The case where $\mathcal{D}$ is a definable class of modules follows by Mehdi and Prest's \cite[Corollary 4.6]{MehdiPrest}, but it is also a consequence of the theory of pp-pairs and pp-formulas (see Prest's \cite[Section 3.4.2]{PrestPurity}).

Definable classes have an equivalent definition in terms of pp-pairs and pp-formulas (see \cite[Theorem 10.1]{PrestMultisorted}), in settings more general than modules over a ring: namely, modules over a skeletally small preadditive category. In particular, pp-pairs and pp-formulas can be  interpreted in the setting of chain complexes, as complexes can be regarded as representations of the quiver 
\[
A^\infty_\infty = \cdots \to \cdot \to \cdot \to \cdot \to \cdots
\]
Thus if $\mathcal{D}$ is a definable category of chain complexes then there will be a set $\Phi$ of pp-pairs such that $\mathcal{D}$ consists of exactly those complexes where all the pairs in $\Phi$ are closed. 

The variables that appear in pp-formulas for complexes are \emph{sorted} in the sense of \cite{PrestMultisorted}, that is, (1) each variable ranges over elements located at a fixed vertex $i$ of the previous quiver, (2) there is a copy of $R$ acting as scalars at each vertex, and (3) there is a \emph{ringoid} (or \emph{ring with several objects}) element for each arrow. An example of a pp-formula is $\exists y_{i+1} \, (x_{i-1} =d_{i}d_{i+1}\,y_{i+1})$, where $d_i$ denotes the arrow $i \to i-1$. In the case of complexes, one has $x_{i-1} = 0$. 

Following this brief explanation of pp-formulas, it can be noted that each pp-formula (resp. each pp-pair) can be written as a pp-formula (or pp-pair) over a ring: namely, the path algebra of a finite portion of the quiver $A_\infty^\infty$ which contains all the vertices corresponding to sorts of variables or end points of arrows appearing. In other words, any pp-formula involves only a finite number of objects of the category, so is essentially over a ring. 

In order to check whether the equivalence $D \in \mathcal{D} \Longleftrightarrow D^{++} \in \mathcal{D}$ is true in Lemma~\ref{lem:Prest}, choose some set $\Phi$ of pp-pairs defining $\mathcal{D}$. Indeed, we know by \cite[Theorem 10.1]{PrestMultisorted} that there is a set $\Phi$ of pp-pairs for complexes such that 
\[
\mathcal{D} = \{ X \in \Ch(R) \mbox{ {\rm : }} \phi(X) / \psi(X) = 0, \mbox{ } \forall \mbox{ } \phi / \psi \in \Phi\}.
\]
Then, for each pair in $\Phi$ we must check that the pair is \emph{closed} (see \cite[Section 6]{PrestMultisorted}) on the double dual. We need to check one pair at a time, but this problem reduces to a problem over a ring $R$ (that is, over a finite-ringoid = 1-sorted ring). So closure on the double dual needs to be checked only on a finite number of the complexes involved in the problem. Hence, the theory of modules over a (1-sorted) ring can be applied, temporarily regarding or replacing the multi-but still finite-sorted pp-formulas by formulas over a normal ring $R$.


\section*{\textbf{Acknowledgements}}

Part of this work was carried out during research stays at the Centre International de Rencontres Math\'ematiques (Luminy, France) and at the Centro Internazionale per la Ricerca Matematica of the Fondazione Bruno Kessler in Trento (Italy), between June and July of 2017 as part of the \emph{Research in Pairs} program. We want to thank the staff from both institutions for their hospitality.

Also, special thanks to Mike Prest, who pointed out the validity of Lemma \ref{lem:Prest}, which was key in presenting many important results also in the context of chain complexes. 

Finally, we thank the referee for her/his corrections and suggestions.


\bibliographystyle{alpha}
\bibliography{biblioGF}
\end{document}